\RequirePackage{rotating}

\documentclass[12pt]{iopart}

\usepackage{algorithm}
\usepackage{algorithmic} 
\usepackage{tikz}
\usepackage{iopams}
\usepackage{rotating}

\expandafter\let\csname equation*\endcsname\relax

\expandafter\let\csname endequation*\endcsname\relax

\usepackage{amsmath}
\usepackage{amsopn}
\usepackage{amsthm}
\usepackage{todonotes}
\usepackage{mathtools}
\usepackage{graphicx}
\usepackage{array}
\usepackage{hyperref}
\usepackage[noabbrev]{cleveref} 
\usepackage{color}
\usepackage{subcaption}

\allowdisplaybreaks 
\crefformat{equation}{(#2#1#3)}
\newtheorem{theorem}{Theorem}[section]
\newtheorem{lemma}[theorem]{Lemma}
\newtheorem{remark}[theorem]{Remark}
\DeclareMathOperator*{\argmin}{argmin}

\begin{document}

\title[Projected Newton Method]{Projected Newton method for noise constrained $\ell_p$ regularization}

\author{J Cornelis \& W Vanroose}

\address{Applied Mathematics Group, Department of Mathematics, University of Antwerp, Middelheimlaan 1, 2020 Antwerp, Belgium}
\ead{jeffrey.cornelis@uantwerp.be}
\vspace{10pt}
\begin{indented}
\item[]\today
\end{indented}

\begin{abstract}
Choosing an appropriate regularization term is necessary to obtain a meaningful solution to an ill-posed linear inverse problem contaminated with measurement errors or noise. 
The $\ell_p$ norm covers a wide range of choices for the regularization term since its behavior critically depends on the choice of $p$ and since it can easily be combined with a suitable regularization matrix. 
We develop an efficient algorithm that simultaneously determines the regularization parameter and corresponding $\ell_p$ regularized solution such that the discrepancy principle is satisfied. We project the problem on a low-dimensional Generalized Krylov subspace and compute the Newton direction for this much smaller problem. We illustrate some interesting properties of the algorithm and compare its performance with other state-of-the-art approaches using a number of numerical experiments, with a special focus of the sparsity inducing $\ell_1$ norm and edge-preserving total variation regularization.  
\end{abstract}

\vspace{2pc}

\section{Introduction}

In this manuscript we are concerned with the $\ell_p$ regularized linear inverse problem
\begin{equation}\label{eq:lp}
x_\alpha = \argmin_{x\in\mathbb{R}^n} \frac{1}{2} ||A x - b||_2^2 + \frac{\alpha}{p} ||L x||_p^p
\end{equation}
for $1\leq p \leq 2$, where $A\in\mathbb{R}^{m\times n}$ is an ill-conditioned matrix that describes a forward operation, for example, modeling some physical process, and the data $b\in\mathbb{R}^{m}$ contains measurement errors or ``noise'', such that $b = b_{ex} + e$. It is well known that for such problems it is necessary to include a regularization term, since the naive solution (i.e. with $\alpha = 0$) is dominated by noise and does not provide any useful information about the ``exact solution'' $x_{ex}$ corresponding to the noiseless data $b_{ex}$, see for instance \cite{doi:10.1137/1.9780898718836}. The norm $||x||_p = \left(\sum_{i=1}^{n} |x_i|^{p}\right)^{\frac{1}{p}}$ denotes the standard $\ell_p$ norm of a vector $x\in\mathbb{R}^{n}$. For $p=2$ we get the standard Euclidean norm and we sometimes use the simplified notation $||x||_2 = ||x||$. The matrix $L\in\mathbb{R}^{s\times n}$ is referred to as the regularization matrix and is not necessarily square or invertible. It can be used to significantly improve the quality of the reconstruction. Regularization refers to the fact that we incorporate some prior knowledge about the exact solution $x_{ex}$ to obtain a well-posed problem. For instance, if we know that the exact solution is smooth, then a good choice is standard form Tikhonov regularization, which corresponds to choosing $L = I_n$ the identity matrix of size $n$ and $p=2$. When we know that a certain transformation of $x_{ex}$ is smooth, we can take $p=2$ and choose a suitable regularization matrix $L\neq I_n$, in which case we refer to \cref{eq:lp} as general form Tikhonov regularization. A popular choice of regularization matrix is, for instance, a finite difference approximation of the first or second derivative operator. A wide range of efficient algorithms have been developed for the standard form and general form Tikhonov problem \cite{GAZZOLA2014180,Cornelis_2020,LAMPE20122845,CALVETTI2000423}. 

The choice $L = I_n$ and $p = 1$ has also received a lot of attention in literature, since it is known that the $\ell_1$ norm induces sparsity in the solution \cite{doi:10.1137/130917673,rodriguez2008efficient,10.1137/S1064827596304010,doi:10.1137/18M1194456,558475}. Total variation regularization \cite{chan2006total,4380459}, which is a popular regularization technique in image deblurring, provides us an example with $L\neq I_n$ and $p=1$. Here the solution $x \in\mathbb{R}^{n}$ is a vector obtained by stacking all columns of a pixel image $X\in\mathbb{R}^{N\times N}$ with $n = N^2$ and the matrix $A$ is a blurring operator. Let us denote the anisotropic total variation function as
\begin{equation} \label{eq:tv}
\text{TV}(x) = \sum_{i,j = 1}^{N} |\partial^{(i,j)}_h(X)| + |\partial^{(i,j)}_v(X)|.
\end{equation}
with finite difference operators in the horizontal and vertical direction given by 
\begin{equation*} \partial^{(i,j)}_h(X) = \left\{ \begin{matrix} X_{i,j+1} - X_{i,j}  & j<N \\ 0 & j=N \end{matrix} \right. \hspace{0.5cm}\text{and}\hspace{0.5cm} \partial^{(i,j)}_v(X) = \left\{ \begin{matrix} X_{i+1,j} - X_{i,j}  & i<N \\ 0 & i=N \end{matrix} \right. . 
\end{equation*}
We can rewrite this in a more convenient way by first writing 
\begin{equation}\label{eq:forward}
D = \begin{pmatrix}
1 & - 1& & & \\ & \ddots&\ddots&&\\ && 1&-1
\end{pmatrix} \in\mathbb{R}^{(N-1) \times N}
\end{equation}
which represents a finite difference approximation of the derivative operator in one dimension.
Let  $\otimes$ denote the Kronecker product. We can compactly write $\text{TV}(x) = ||L x||_1$ with 
\begin{equation}\label{eq:ltv}
L = \begin{pmatrix}D_h \\D_v \end{pmatrix} \in\mathbb{R}^{(2n-2N) \times n} \hspace{0.5cm}\text{and} \hspace{0.5cm}  \left\{ \begin{matrix} D_h =  D  \otimes I_N \in \mathbb{R}^{(n - N) \times n} \\
D_v = I_N \otimes D \in \mathbb{R}^{(n - N) \times n} \end{matrix}\right. . 
\end{equation}
The matrices $D_h$ and $D_v$ represent the two dimensional finite difference approximation of the derivative operator in the horizontal and the vertical direction respectively.

The scalar $\alpha\in\mathbb{R}$ in \cref{eq:lp} is the regularization parameter and has a huge impact on the quality of the reconstruction. If this value is too small, then the solution $x_\alpha$ closely resembles the naive solution and is ``over-fitted'' to the noisy data $b$. On the other hand, if $\alpha$ is too large, then $x_\alpha$ is not a good solution to the inverse problem anymore and is, for instance, in the case of Tikhonov regularization ``over-smoothed''.  Different parameter choice methods exist for choosing a suitable $\alpha$. One of the most straightforward ways to choose the regularization parameter is given by the discrepancy principle, which states that we should choose $\alpha$ such that   
\begin{equation} \label{eq:discrepancy}
||Ax_{\alpha} - b||_2 = \eta ||e||_2,
\end{equation}
where $\eta\geq 1$ is a safety factor. Obviously $||e||_2$ is not available in practice, so this approach assumes we have some estimate of this value available.

Note that \cref{eq:lp} is a convex optimization problem, which means that any local solution is also a global solution. However, for $p<2$ the problem is non-differentiable. Hence, we cannot use algorithms that rely on the gradient or Hessian of the objective function, like steepest descent or Newton's method \cite{nocedal2006numerical}. To overcome this issue, we simply use a smooth approximation of the regularization term $||Lx||_{p}^{p}$ when $p<2$. 
 
In this paper we develop an algorithm that simultaneously solves (a smooth approximation of) the $\ell_p$ regularized problem \cref{eq:lp} and determines the corresponding regularization parameter $\alpha$ such that the discrepancy principle \cref{eq:discrepancy} is satisfied. The problem is reformulated as a constrained optimization problem for which the solution satisfies a system of nonlinear equations. Newton's method can be used to solve this problem \cite{landi2008lagrange}.

However, this approach can be quite computationally expensive for
large-scale problems since each iteration requires the solution of a
large linear system to obtain the Newton direction. A Krylov subspace
method is typically used to solve the linear system, leading to an
expensive outer-inner iteration scheme, i.e. each outer Newton
iteration requires a number of inner Krylov subspace iterations. We
circumvent this by projecting the constrained optimization problem on
a low-dimensional Generalized Krylov subspace and we calculate the
Newton direction for this projected problem. In each iteration of the
algorithm we use this search direction in combination with a
backtracking line search, after which the Generalized Krylov subspace
is expanded.  Further improvements to the algorithm are presented in
the case of general form Tikhonov regularization. This newly developed
algorithm can be seen as a generalization of the Projected Newton
method for standard form Tikhonov regularization
\cite{Cornelis_2020}. In fact, some results from  \cite{Cornelis_2020}
are extended and proven in a more general context.

We compare our method with the Generalized Krylov subspace (GKS) method developed in \cite{LAMPE20122845}, which can be used to solve the general form Tikhonov problem. In addition, we also compare the performance of the Projected Newton method with that of the GKSpq method developed in \cite{doi:10.1137/140967982} by a number of timing experiments. The latter algorithm combines the well-established Iteratively reweighted norm approach \cite{rodriguez2008efficient,4380459} with a projected step on a low-dimensional Generalized Krylov subspace. We perform experiments using the sparsity inducing $\ell_1$ norm as well as the total variation regularization term.

The paper is organized as follows. In \cref{sec:smooth_approx} we
define a smooth approximation of the $\ell_p$ norm for $p<2$ and in \cref{sec:newton} we formulate the nonlinear
system of equations that describes the problem of interest.  The main
contribution of this paper is presented in \cref{sec:main}, where we
develop the Projected Newton method and prove our main
results. \Cref{sec:refmeth} describes two reference methods that we
use to compare the Projected Newton method with. Next, in
\cref{sec:num} we provide a number of experiments illustrating the
performance and overall behavior of the newly proposed
algorithm. Lastly, this work is concluded in \cref{sec:concl}.
 
\section{Smooth approximation of the $\ell_p$ norm} \label{sec:smooth_approx}

The $\ell_p$ norm is non-differentiable for $p<2$, which makes the
optimization problem a bit more difficult. However, it is easy to
formulate a smooth approximation $\Psi_p(x) \approx
\frac{1}{p}||x||^{p}_p$, where $\Psi_p: \mathbb{R}^n\rightarrow
\mathbb{R}$ is a twice continuously differentiable convex function. More precisely we define
\begin{equation}\label{eq:smoothlp}
\Psi_p(x) = \frac{1}{p}\sum_{i=1}^{n}(x_{i}^2 + \beta)^{\frac{p}{2}},
\end{equation}
where $\beta > 0$ is a small scalar that ensures smoothness. Other possible smooth approximations of the absolute value function can alternatively be chosen, see for instance \cite{saheya2018numerical,herty2007smoothed,wu2019signal}.  
The gradient $\nabla \Psi_p(x)=\left(\frac{\partial \Psi_p(x)}{\partial x_1} ,\ldots,\frac{\partial \Psi_p(x)}{\partial x_n}\right)^T$ is given by the partial derivatives
\begin{equation*}
\frac{\partial \Psi_p(x)}{\partial x_i} = x_i (x_i^2 + \beta)^{\frac{p}{2}-1}
\end{equation*}
for $i=1,\ldots,n$. The Hessian matrix $\nabla^{2}\Psi_p(x) = \left( \frac{\partial^2 \Psi_p(x)}{\partial x_i\partial x_j}  \right)_{i,j = 1,\ldots,n}$ is a diagonal matrix since $\frac{\partial \Psi_p(x)}{\partial x_i}$ does not contain any $x_j$ with $i\neq j$ and is thus given by the following second derivatives: 
\begin{equation*}
\frac{\partial^2 \Psi_p(x)}{\partial x_i^2} = (x_i^2 + \beta)^{\frac{p}{2}-1} + 2 (\frac{p}{2}-1)x_i^2 (x_i^2 + \beta)^{\frac{p}{2}-2} > 0.
\end{equation*}
From this it also follows that $\Psi_p(x)$ is strictly convex. The following lemma can be used to calculate the gradient and Hessian for the smooth approximation $\Psi_{p}(Lx) \approx \frac{1}{p}||Lx||_p^p$. 

\begin{lemma} \label{thm:chain}
Let $L\in\mathbb{R}^{s \times n}$ and $x\in\mathbb{R}^n$. Let $\tilde{\Psi}(z):\mathbb{R}^{s} \longrightarrow \mathbb{R}$ be a twice continuously differentiable function and $\Psi(x) = \tilde{\Psi}(Lx)$, then we have
\begin{align*}
\nabla \Psi (x) &=  L^T \nabla \tilde{\Psi} (L x), \\
\nabla^2 \Psi (x) &=  L^T \nabla^2 \tilde{\Psi} (L x ) L.
\end{align*}
\end{lemma}
\begin{proof}
This follows from the chain rule for multivariate functions. 
\end{proof}

\section{Reformulation of the problem}\label{sec:newton}

For the moment, let us denote $\Psi(x) : \mathbb{R}^n \longrightarrow \mathbb{R}$ any twice continuously differentiable convex function. For instance, for the smooth approximation of the regularization term $\frac{1}{p}||Lx||_p^p$ with $1\leq p < 2$ we take $\Psi(x)=\Psi_p(Lx)$, while for $p=2$ we simply take the actual regularization term $\Psi(x) = \frac{1}{2}||Lx||_2^2$ since it is already smooth.  The goal is to develop an efficient algorithm that can simultaneously solve the convex optimization problem
\begin{equation}\label{eq:smooth}
x_\alpha = \argmin_{x\in\mathbb{R}^n} \frac{1}{2} ||A x - b||_2^2 + \alpha \Psi(x),
\end{equation}
and find the corresponding regularization parameter $\alpha$ such that the discrepancy principle \cref{eq:discrepancy} is satisfied. Note that this goal is slightly more general than expressed in the introduction. The (possibly nonlinear) system of equations
\begin{equation}\label{eq:smooth_cond}
A^T(Ax - b) + \alpha \nabla \Psi(x) = 0
\end{equation}
are necessary and sufficient conditions for $x$ to be a global solution for \cref{eq:smooth} due to convexity of the objective function. Uniqueness of the solution is in general not guaranteed. For instance, if we want the ensure a unique solution for general form Tikhonov regularization we can add the requirement that $\mathcal{N}(A)\cap\mathcal{N}(L)= \{0\}$, where $\mathcal{N}(\cdot)$ denotes the null-space of a matrix.   
In fact, the following more general result holds
\begin{lemma} \label{thm:unique}
Let $\Psi(x) = \tilde{\Psi}(Lx)$ with a twice continuously differentiable strictly convex function $\tilde{\Psi}(z):\mathbb{R}^s \longrightarrow \mathbb{R}$ and $L\in\mathbb{R}^{s\times n}$ such that $\mathcal{N}(A)\cap\mathcal{N}(L)= \{0\}$ then it holds that \cref{eq:smooth} has a unique solution.
\end{lemma}
\begin{proof}
Let $x_1$ and $x_2$ be two different solutions of \cref{eq:smooth}. Suppose that both $Lx_1=Lx_2$ and $Ax_1 = Ax_2$ hold. This would imply that $x_1 - x_2 \in \mathcal{N}(A)\cap\mathcal{N}(L)$, meaning $x_1 = x_2$, which contradicts the assumption that we have two different solutions. Hence, at least one of these should be an inequality. Let us denote $x_3 = (x_1 + x_2)/2$. Suppose $Lx_1 \neq Lx_2$, then we have that $\tilde{\Psi}(Lx_3) < \frac{1}{2}\tilde{\Psi}(Lx_1) +  \frac{1}{2}\tilde{\Psi}(Lx_2)$ since $\tilde{\Psi}$ is strictly convex. If $Lx_1 = Lx_2$ then we must have  $Ax_1 \neq Ax_2$ and we have in this case $||Ax_3 - b||^2 < \frac{1}{2}||Ax_1 - b||^2 + \frac{1}{2}||Ax_2 - b||^2 $ since the squared Euclidean norm is also strictly convex. Hence, at least one of these inequalities is strict. This implies that 
\begin{equation*}
\frac{1}{2}||A x_3- b||_2^2 + \alpha\tilde{\Psi}(Lx_3) <  \frac{1}{4} ||A x_1 - b||_2^2 + \frac{\alpha}{2}\tilde{\Psi}(L x_1) +  \frac{1}{4} ||A x_2 - b||_2^2 + \frac{\alpha}{2}\tilde{\Psi}(L x_2). 
\end{equation*}
This leads to a contradiction since every local solution is a global solution and we have found a point $x_3$ with strictly smaller objective value than $x_1$ and $x_2$.
\end{proof}
\noindent Recall that $\tilde{\Psi}(z) = \Psi_p(z)$ is strictly convex and thus it follows from this lemma that the smooth approximation of the $\ell_p$ regularized problem has a unique solution if $\mathcal{N}(A)\cap\mathcal{N}(L)= \{0\}$. 

Let us now turn our attention to the following constrained optimization problem
\begin{equation} \label{eq:reform}
\min_{x\in\mathbb{R}^n} \hspace{0.2cm} \Psi(x) \hspace{1cm} \text{subject to} \hspace{1cm} \frac{1}{2} ||Ax - b||^{2}_2 = \frac{\sigma^2}{2},
\end{equation}
where we denote $\sigma = \eta ||e||_2$ the value used in the discrepancy principle \cref{eq:discrepancy}. We assume that $||b|| > \sigma$, since otherwise the noise is larger than the data and there is no use in trying to solve the inverse problem. This optimization problem is closely related to \cref{eq:smooth}. A solution $(x^{*},\lambda^{*})$ to \cref{eq:reform} has to satisfy the nonlinear system of equations $F(x^*,\lambda^*) = 0$ with
\begin{equation} \label{eq:bigf}
F(x,\lambda) = \begin{pmatrix}
 \lambda A^T(Ax-b)+ \nabla \Psi(x) \\
\frac{1}{2} ||Ax - b||^{2}_2 - \frac{\sigma^2}{2}\end{pmatrix} \in \mathbb{R}^{n + 1}.
\end{equation}
These express the first order optimality conditions of problem \eqref{eq:reform}, also known as Karush-Kuhn-Tucker or KKT-conditions \cite{nocedal2006numerical}. The author in \cite{landi2008lagrange} showed, under the assumption that there exists a constant $c\geq 0$ such that for all $x$ it holds that $\Psi(x)\geq c,\Psi(0) = c$ and that $\{x : \Psi(x) = c\} \cap \{ x : ||Ax-b||\leq \sigma\} = \emptyset$, that they are sufficient conditions for $(x^{*},\lambda^{*})$ to be a global solution and that the Lagrange multiplier $\lambda^{*}$ is strictly positive. Note that this assumption holds for $\Psi(x) = ||x||^2_2$ with $c=0$ and for $\Psi(x) = \Psi_p(x)$ with $c=n\beta^{\frac{p}{2}}/p$ since for these choices we have $\{x : \Psi(x) = c\} = \{ 0\}$; and since $||b||>\sigma$ the intersection is indeed empty. If we consider $\Psi(x) = ||L x||^2_2$ and $\Psi(x) = \Psi_p(Lx)$, we can again take the corresponding value for $c$ and we have $\{x : \Psi(x) = c\} = \mathcal{N}(L)$. However, in general we cannot say much about $ \mathcal{N}(L)\cap \{ x : ||Ax-b||\leq \sigma\} $.

Now, if $(x,\lambda)$ is any root of the first component of \cref{eq:bigf} with $\lambda>0$ it is also a solution to \cref{eq:smooth} for $\alpha = 1/\lambda$, which follows from the fact that $(x,\alpha)$ solves the optimality conditions \cref{eq:smooth_cond} in that case. This means that if we solve \cref{eq:reform} we simultaneously solve the regularized linear inverse problem \cref{eq:smooth} and find the corresponding regularization parameter $(\alpha = 1/\lambda)$ such that the discrepancy principle is satisfied. Due to the straightforward connection between $\alpha$ and $\lambda$ we refer to both quantities as the regularization parameter. Uniqueness of the solution of \cref{eq:reform} can be proven under the same conditions as \cref{thm:unique} using similar arguments as presented in its proof.

The Newton direction for the nonlinear system of equations $F(x,\lambda)=0$ given by \cref{eq:bigf} in a point $(x,\lambda)$ is given by the solution of the linear system
\begin{equation} \label{eq:bignewt}
J(x,\lambda) \begin{pmatrix} \Delta x \\ \Delta \lambda \end{pmatrix} = -F(x,\lambda), 
\end{equation}
where the Jacobian $J(x,\lambda) \in \mathbb{R}^{(n+1)\times (n+1)}$ of the function $F(x,\lambda)$ is given by
\begin{equation}\label{eq:jacbig}
J(x,\lambda) =\begin{pmatrix}\lambda  A^T A + \nabla^{2} \Psi(x)  &  A^T(Ax-b) \\ (Ax-b)^T A & 0 \end{pmatrix}.
\end{equation}

This linear system \cref{eq:bignewt} cannot be solved using a direct method when $n$ is very large. Moreover, for many applications the matrix $A$ is not explicitly given and we can only compute matrix-vector products with $A$ and $A^T$. Hence, a Krylov subspace method such as MINRES \cite{doi:10.1137/0712047} is used to compute the Newton direction. However, every iteration of the linear solver requires a matrix-vector product with $A$ and $A^T$, which becomes  expensive when a lot of iterations need to be performed. In the following section we develop the  alternative approach proposed in this paper.

\section{Projected Newton method}\label{sec:main}
In this section we derive a	Newton-type method that solves \cref{eq:reform} and that only requires one matrix-vector product with $A$ and one matrix-vector product with $A^T$ in each iteration.

\subsection{Projected minimization problem} \label{sec:projminsec}
Suppose we have a matrix $V_k = [v_0,v_1,\ldots,v_{k-1}] \in \mathbb{R}^{n\times k}$ with orthonormal columns, i.e. $V_k^T V_k = I_k$. The index $k$ is the iteration index of the algorithm that we describe in this section. Here, and in what follows, a sub-index refers to a certain iteration number rather than, for instance, an element of a vector.  Let $\mathcal{R}(V_k)$ denote the range of the matrix $V_k$, i.e. the space spanned by all columns of the matrix. We consider the $k$-dimensional \textit{projected minimization problem} 
\begin{align} \label{eq:projmin}
& \min_{x\in\mathcal{R}(V_{k})} \hspace{0.2cm} \Psi(x) \hspace{1cm} \text{subject to} \hspace{1cm} \frac{1}{2} ||Ax - b||^{2}_2 = \frac{\sigma^2}{2}  \nonumber \\
\Leftrightarrow & \min_{y\in \mathbb{R}^k} \hspace{0.2cm} \Psi(V_k y) \hspace{1cm} \text{subject to} \hspace{1cm}  \frac{1}{2} ||AV_k y - b||^{2}_2 = \frac{\sigma^2}{2}. 
\end{align}
The corresponding KKT conditions are now given by
\begin{equation}\label{eq:projfunction}
F^{(k)}(y,\lambda) = 
\begin{pmatrix}
 \lambda V_k^T A^T (AV_k y - b) + V_k^T \nabla \Psi(V_k y) \\
\frac{1}{2} ||AV_k y - b||^{2}_2 - \frac{\sigma^2}{2} \end{pmatrix} = \begin{pmatrix} V_{k}^T & 0 \\ 0 & 1 \end{pmatrix} F(V_k y,\lambda),
\end{equation}
which can be seen as a projected version of \cref{eq:bigf}. 
The Jacobian of the \textit{projected function} $F^{(k)}(y,\lambda) \in\mathbb{R}^{k + 1}$, which we refer to as the \textit{projected Jacobian},  is given by
\begin{equation} \label{eq:jacobian}
J^{(k)}(y,\lambda)  = \begin{pmatrix} \lambda V_k^T A^T A V_k + V_k^T\nabla^2 \Psi(V_k y) V_k & V_k^T A^T (AV_k y - b) \\  (AV_k y - b)^TAV_k & 0 \end{pmatrix} \in\mathbb{R}^{(k + 1)\times (k + 1)}.
\end{equation}
We have the following connection between the Jacobian \cref{eq:jacbig} and projected Jacobian
\begin{equation} \label{eq:vjv}
 J^{(k)}(y,\lambda) = \begin{pmatrix} V_{k}^T & 0 \\ 0 & 1 \end{pmatrix} J(V_k y,\lambda) \begin{pmatrix} V_{k} & 0 \\ 0 & 1 \end{pmatrix}. 
\end{equation}

Let us denote $\bar{y}_{0}= 0$ and $\bar{y}_{k-1} = (y_{k-1}^T,0)^T \in \mathbb{R}^{k}$ for $k>1$ where $y_{k-1}$ is an approximate solution for the $(k -1)$-dimensional minimization problem, i.e. problem \cref{eq:projmin} with $k$ replaced by $k-1$. Let us furthermore write $x_{k} = V_{k}y_{k}$ for all $k>0$ and $x_0 = 0$. Since the last component of $\bar{y}_{k-1}$ is zero we have $V_{k}\bar{y}_{k - 1} = V_{k - 1}y_{k-1} = x_{k - 1}$ for $k>1$ and $V_1 \bar{y}_0 = x_0 = 0$ by definition. 
If $J^{(k)}(\bar{y}_{k-1},\lambda_{k-1})$ is nonsingular we can calculate the Newton direction for the projected problem as the solution of the linear system
\begin{equation} \label{eq:newtoneq}
 J^{(k)}(\bar{y}_{k-1},\lambda_{k-1}) \begin{pmatrix} \Delta y_{k}\\ \Delta \lambda_{k} \end{pmatrix} = - F^{(k)}(\bar{y}_{k-1},\lambda_{k-1}).
\end{equation}
We refer to this as the \textit{projected Newton direction}. For a suitably chosen step-length $0<\gamma_k \leq 1$, we can update our sequence by 
\begin{equation*}
y_{k} = \bar{y}_{k-1} + \gamma_k  \Delta y_{k}, \hspace{0.5cm}\lambda_{k} = \lambda_{k-1} + \gamma_k \Delta \lambda_{k}.
\end{equation*}
This gives us a corresponding update for $x_{k}$:
\begin{equation*}
x_{k} = V_{k} y_{k} = V_{k}\bar{y}_{k-1} + \gamma_k  V_{k}\Delta y_{k} = x_{k-1} + \gamma_k \Delta x_{k},
\end{equation*}
where we define $\Delta x_k = V_{k}\Delta y_{k}$. Note that this step is different from the step that would be obtained by calculating the true Newton direction $J(x_{k-1},\lambda_{k-1})^{-1}F(x_{k-1},\lambda_{k-1})$. However, if we choose a particular basis $V_k$, we will show that this provides us with a descent direction for the merit function $f(x,\lambda)= \frac{1}{2} ||F(x,\lambda)||^2$. 

In general we can not conclude from \cref{eq:projfunction} that 
\begin{equation}\label{eq:fk2}
\begin{pmatrix} V_{k} & 0 \\ 0 & 1 \end{pmatrix}F^{(k)}(\bar{y}_{k-1},\lambda_{k-1}) =  F(V_{k}\bar{y}_{k-1},\lambda_{k-1}).
\end{equation}
However, by considering a specific choice of basis $V_k$ we can enforce this equality. Recall that $V_{k}\bar{y}_{k-1} = x_{k-1}$. Let us denote 
\begin{equation} \label{eq:vtilde}
\tilde{v}_{k - 1} = \lambda_{k-1} A^T(Ax_{k-1} - b) + \nabla \Psi(x_{k-1})
\end{equation}
the first component of $F(x_{k-1},\lambda_{k-1})$, see \cref{eq:bigf}. 
Equation \eqref{eq:fk2} holds if and only if
\begin{equation*}
V_k V_k^T \tilde{v}_{k-1} = \tilde{v}_{k-1}
\end{equation*}
which is true if $\tilde{v}_{k-1} \in \mathcal{R} (V_{k})$. 

Now we have a straightforward way to construct the basis $V_k = [v_0,\ldots,v_{k-1}]$ such that \cref{eq:fk2} holds. Let $V_1 = \tilde{v}_0/ ||\tilde{v}_0||$, with $\tilde{v}_0 = -\lambda_0 A^T b$ since $x_0=0$ and $\nabla\Psi(0)=0$. 
Suppose we already have the basis $V_{k}$ in iteration $k$ and have just constructed new variables $\lambda_{k}$ and $x_{k}$. To add a new vector $v_k$ to get $V_{k+1}$, we simply take $\tilde{v}_k$ and orthogonalize it to all previous vectors in $V_{k}$ using Gram-Schmidt and then normalize, i.e: 
\begin{equation} \label{eq:gs}
v_{k} = \tilde{v}_{k} - \sum_{j=0}^{k-1} (v_{j}^T \tilde{v}_{k}) v_{j},\hspace{0.5cm} v_{k} = v_{k}/||v_{k}||. 
\end{equation}
Note that we use modified Gram-Schmidt in the actual implementation of the algorithm, but for notational convenience we write in the classical way. By construction we now have that $V_{k}$ has orthonormal columns and
\begin{equation} \label{eq:genkryl}
\mathcal{R}(V_{k + 1}) = \text{span} \left\{\tilde{v}_0,\tilde{v}_1,\ldots,\tilde{v}_{k}  \right\}.
\end{equation}
The basis $V_k$ is unique up to sign change of each of the vectors.

First note that when the $k$-dimensional projected minimization problem \cref{eq:projmin} is solved for $y = y_{k}$, i.e. if $||F^{(k)}(y_k,\lambda_k)||=0$, this does not necessarily imply that problem \cref{eq:reform} is solved for $x_k = V_ky_k$.
However, in that case we have that $\lambda_k V_k^T A^T (Ax_k - b) + V_k^T \nabla \Psi(x_k)=0$. This implies that $\tilde{v}_{k}$ is already orthogonal to $V_k$ and then the Gram-Schmidt procedure is in principle not necessary. In this case, since in practice we are always working with finite precision arithmetic, \cref{eq:gs} can be seen as a re-orthogonalization step.

Secondly, we note that the basis $V_k$ can be seen as a Generalized Krylov subspace, similarly as in \cite{LAMPE20122845}. Indeed, let us denote the Krylov subspace of dimension $k$ for $M\in\mathbb{R}^{n\times n}$ and $v\in\mathbb{R}^n$ as $\mathcal{K}_k(M,v) = \text{span}\left\{v,Mv,\ldots,M^{k-1}v\right\}$. Consider the case when $\Psi(x) = \frac{1}{2}||x||^2$, i.e. standard form Tikhonov regularization. Then we have 
\begin{equation*}
\tilde{v}_{k - 1} = \lambda_{k-1} A^T(Ax_{k-1} - b) + x_{k-1} =  (\lambda_{k-1} A^T A + I_n) x_{k-1} - A^T b.
\end{equation*}
In particular for $x_0=0$ we have $v_{0} = \pm A^Tb / ||A^T b||$. Now, due to the shift invariance of Krylov subspaces, i.e. the fact that
\begin{equation}\label{eq:shiftinv}
\mathcal{K}_k(A^TA,A^Tb) = \mathcal{K}_k(A^TA + \alpha I, A^T b), \hspace{1cm} \forall \alpha\in\mathbb{R}
\end{equation}
it can easily be seen that $\mathcal{R}(V_k)=\mathcal{K}_k(A^TA,A^Tb)$. As a consequence, we now also have that $V_k^T A^T A V_k$ is a tridiagonal matrix. This is because the basis $V_k$ is (up to sign change of the vectors) the same basis as generated by the Golub-Kahan bidiagonalization procedure \cite{paige1982lsqr,golub1965calculating}. This basis satisfies a relation of the form $AV_k = U_{k+1}B_{k+1,k}$ with a upper bidiagonal matrix $B_{k+1,k} \in\mathbb{R}^{(k+1) \times k}$ and matrix $U_{k+1} \in\mathbb{R}^{m \times n}$ with orthonormal columns. Hence, we get $V_k^TA^TAV_k = B_{k+1,k}^T B_{k+1,k}$, which is indeed tridiagonal. 
Note that the Golub-Kahan bidiagonalization procedure is used in the Projected Newton method for standard form Tikhonov regularization \cite{Cornelis_2020}. Hence, the results presented in this section can be seen as a generalization of some of the results in \cite{Cornelis_2020}. In fact, the algorithm presented in the current paper, when applied to the standard form Tikhonov problem, is (in exact arithmetic) equivalent to the algorithm presented in \cite{Cornelis_2020}, although implemented in a different way.  

When we consider, for instance, general form Tikhonov regularization, i.e.  $\Psi(x) = \frac{1}{2}||Lx||^2$ with $L\neq I_n$ we do not have in general that $V_{k}$ spans a Krylov subspace. This would only be true if the regularization parameter $\lambda_{k}$ remains the same for all $k$. The interesting thing to note, however, is that if the parameter $\lambda_{k}$ stabilizes quickly, then we can recognize an approximate tridiagonal structure in $V_{k}^T ( A^TA + \alpha_k L^T L) V_k$ with $\alpha_k = 1/\lambda_k$. More precisely, the elements on the three main diagonals have much larger magnitude than all other elements. We discuss this property in a bit more detail in \cref{sec:num} and illustrate it with a small numerical experiment (see experiment 1).

Using the Generalized Krylov subspace basis $V_k$ as described above, we can prove the following theorem:
\begin{theorem} \label{thm:descent} Let $V_k$ be the Generalized Krylov subspace basis \cref{eq:genkryl} and $(\Delta y_k^T,\Delta \lambda_k)^T$ be the Projected Newton direction obtained by solving \cref{eq:newtoneq}. Then the step $\Delta_{k} = ( \Delta x_{k}^T , \Delta \lambda_{k})^T $ with $\Delta x_k = V_k \Delta y_k$ is either a descent direction for $f(x_{k-1},\lambda_{k-1})$, i.e.
\begin{equation*}
\Delta_{k}^T \nabla f(x_{k-1},\lambda_{k-1}) < 0
\end{equation*}
or we have found a solution to \cref{eq:reform}.
\end{theorem}
\begin{proof}
Using \cref{eq:vjv}, which holds for any matrix $V_k$, and \cref{eq:fk2}, which holds for the Generalized Krylov subspace basis $V_k$, together with the definition of the step $\Delta_k$ we have: 
\begin{align*}
&\Delta_{k}^T \nabla f(x_{k-1},\lambda_{k-1}) = \begin{pmatrix} \Delta x_{k} \\ \Delta \lambda_{k} \end{pmatrix}^T J(x_{k-1},\lambda_{k-1})F(x_{k-1},\lambda_{k-1}) \\
=& \begin{pmatrix} V_k\Delta y_{k} \\ \Delta \lambda_{k} \end{pmatrix}^T J(V_{k}\bar{y}_{k-1},\lambda_{k-1})F(V_{k}\bar{y}_{k-1},\lambda_{k-1}) \\
%=& \begin{pmatrix} \Delta y_{k} \\ \Delta \lambda_{k} \end{pmatrix}^T \begin{pmatrix} V_{k}^T & 0 \\ 0 & 1 \end{pmatrix} J(V_{k}\bar{y}_{k-1},\lambda_{k-1})F(V_{k}\bar{y}_{k-1},\lambda_{k-1}) \\
=& \begin{pmatrix} \Delta y_{k} \\ \Delta \lambda_{k} \end{pmatrix}^T \begin{pmatrix} V_{k}^T & 0 \\ 0 & 1 \end{pmatrix} J(V_{k}\bar{y}_{k-1},\lambda_{k-1}) \begin{pmatrix} V_{k} & 0 \\ 0 & 1 \end{pmatrix} F^{(k)}(\bar{y}_{k-1},\lambda_{k-1})  \\
=& \begin{pmatrix} \Delta y_{k} \\ \Delta \lambda_{k} \end{pmatrix}^T J^{(k)}(\bar{y}_{k-1},\lambda_{k-1})  F^{(k)}(\bar{y}_{k-1},\lambda_{k-1})\\
=& -  \left(J^{(k)}(\bar{y}_{k-1},\lambda_{k-1}) ^{-1} F^{(k)}(\bar{y}_{k-1},\lambda_{k-1})\right)^T J^{(k)}(\bar{y}_{k-1},\lambda_{k-1})  F^{(k)}(\bar{y}_{k-1},\lambda_{k-1})\\
=& - ||F^{(k )}(\bar{y}_{k-1},\lambda_{k-1}) ||^2  = - ||F(x_{k-1},\lambda_{k-1}) ||^2  \leq 0. 
\end{align*}
%The last equality follows from the fact that $V_k$ has orthonormal columns.
The proof now follows from the fact that $(x_{k-1},\lambda_{k-1})$ is a solution to \cref{eq:reform} if and only if $||F(x_{k-1},\lambda_{k-1}) ||  = 0$.
\end{proof}
The importance of the above theorem is illustrated by the fact that for $\gamma_k$ small enough we have by Taylor's theorem \cite{nocedal2006numerical,boyd2004convex} that
\begin{equation*}
f(x_{k - 1} + \gamma_k \Delta x_k , \lambda_{k-1} + \gamma_k\Delta \lambda_k) \approx f(x_{k-1},\lambda_{k-1}) +\gamma_k\Delta_k^T \nabla f(x_{k-1},\lambda_{k-1})
\end{equation*}
which implies we can find a step-length $\gamma_k$ such that we have a strict decrease in the merit function $f(x_{k},\lambda_{k})<f(x_{k-1},\lambda_{k-1})$. 
In practice, a so-called backtracking line search is often used to find a step-length $\gamma_k > 0$ such that there is a ``sufficient decrease'' of the merit function: 
\begin{equation} \label{eq:sufficient_decrease}
\frac{1}{2} ||F(x_{k},\lambda_{k})||^2 \leq \left(\frac{1}{2} - c\gamma_k \right) ||F(x_{k - 1},\lambda_{k - 1})||^2
\end{equation}
with $c\in (0,1)$. Equation \cref{eq:sufficient_decrease} is often referred to in literature as the sufficient decrease condition or Armijo condition. To find such a step-length, we simply start with $\gamma_k = 1$ and check if the sufficient decrease condition holds. If not, we reduce $\gamma_k$ by a factor $0<\tau<1$ and check if $\gamma_k := \tau \gamma_k$, satisfies the condition. This procedure is then repeated until a suitable step-length is found. 
We obviously do not want to calculate the new norm $F(x_{k},\lambda_{k})$ by performing matrix-vector products with $A$ and $A^T$. This would make the line search expensive if the step-length is reduced multiple times. In \cref{sec:improvementsback} we describe how this line search can be performed efficiently. The basic idea of the proposed method is now given by \cref{alg:PN_basic}.

\begin{algorithm}
\caption{Projected Newton method (basic framework)}
\label{alg:PN_basic} 
\begin{algorithmic}[1] \small
\STATE{Initialize $x_0=0, V_1 = A^Tb/||A^T b||$ and choose $\lambda_0 > 0$. \label{init}}
\FOR{$k = 1,2,\ldots,\text{until convergence}$} 
\STATE{Calculate descent direction $\Delta_k = (\Delta x_k^T,\Delta \lambda_k)^T$ with $\Delta x_k = V_k\Delta y_k$ using \cref{eq:newtoneq}.}
\STATE{Choose step-length $\gamma_k$ such that sufficient decrease condition \cref{eq:sufficient_decrease} is satisfied.}
\STATE{Expand Generalized Krylov subspace basis $V_{k+1} = [V_k,v_k]$ using \cref{eq:vtilde} and \cref{eq:gs}.}
\ENDFOR
\end{algorithmic}
\end{algorithm}

In \cref{sec:num} we comment on the different criteria that can be used to check for convergence. 
The following interesting result will be useful in this discussion.
\begin{lemma}\label{thm:posdis}
For all $k\geq 0$ we have that the iterates $x_k$  generated by \cref{alg:PN_basic} satisfy $||Ax_k - b||\geq \sigma$, which means we never get a residual smaller than what the discrepancy principle dictates.  
\end{lemma}
\begin{proof}
We prove this by induction. For $k=0$ we have by assumption that $||b|| \geq \sigma$. Suppose $k>0$ and that $||Ax_{k-1} - b||=||AV_k\bar{y}_{k-1} - b||\geq \sigma $.  Writing out the last component of the equality \cref{eq:newtoneq} we get
\begin{equation}\label{eq:lastcomp}
(AV_{k}\bar{y}_{k-1} - b)^T AV_{k} \Delta y_{k} = \frac{\sigma^2}{2} - \frac{1}{2} ||AV_{k}\bar{y}_{k-1} - b||^2 \leq 0.
\end{equation}
 Now the proof follows from the following calculation:
\begin{align*} 
&||Ax_k - b||^2 = ||AV_{k}y_{k} - b||^2 = ||AV_{k}(\bar{y}_{k-1} + \gamma_{k} \Delta y_{k}) - b||^2 \\
&= ||AV_{k}\bar{y}_{k-1} - b ||^2 + \gamma_{k}^2 ||AV_{k}\Delta y_{k}||^2 + 2 \gamma_{k} (AV_{k}\bar{y}_{k-1} - b)^T AV_{k} \Delta y_{k} \\
&\geq ||AV_{k}\bar{y}_{k-1} - b ||^2 + \gamma_{k}^2 ||AV_{k}\Delta y_{k}||^2 + 2(AV_{k}\bar{y}_{k-1} - b)^T AV_{k} \Delta y_{k} \\
&= ||AV_{k}\bar{y}_{k-1} - b ||^2 + \gamma_{k}^2 ||AV_{k}\Delta y_{k}||^2 + \sigma^2 - ||AV_{k}\bar{y}_{k-1} - b||^2 \\
&= \sigma^2 + \gamma_{k}^2 ||AV_{k}\Delta y_{k}||^2. 
\end{align*}
The inequality follows from the fact that $(AV_{k}\bar{y}_{k-1} - b)^T AV_{k} \Delta y_{k}\leq 0$ by \cref{eq:lastcomp} and $0 < \gamma_k \leq 1$.
\end{proof}

\subsection{Efficiently computing the projected Newton direction}
In this section we efficiently construct the projected Jacobian \cref{eq:jacobian} and projected function \cref{eq:projfunction} that we need to compute the projected Newton direction \cref{eq:newtoneq}. To do so, we consider the reduced QR decomposition  of the tall and skinny matrix $AV_{k} \in \mathbb{R}^{m \times k}$, which was also the approach taken in \cite{LAMPE20122845,doi:10.1137/140967982}.  
Let $Q_k  \in \mathbb{R}^{m \times k}$ with $Q_k^TQ_k = I_k$  and $R_k \in \mathbb{R}^{k \times k}$ upper triangular such that $AV_k = Q_k R_k$. Using this QR decomposition we can write $V_k^T A^T A V_k =  R_k^T R_k$ and as a consequence we also have
\begin{equation}\label{eq:thisvec}
V_k^T A^T (A V_k y - b) = R_k^T R_k y - V_k^T A^T b = R_k^T R_k y - d_k
\end{equation}
with $d_k = V_k^T A^T b =  ||A^Tb|| e^{(k)}_1$ and $e^{(k)}_1 = (1,0,\ldots,0)^T \in \mathbb{R}^k$. 
The vector \cref{eq:thisvec} is present in both the projected Jacobian \cref{eq:jacobian} and the projected function \cref{eq:projfunction}.
The QR decomposition of $AV_k$ can be efficiently updated in each iteration. For $k=1$ we trivially have $AV_1$ = $Q_1 R_1$ with $R_1 = ||AV_1||$ and $Q_1 = AV_1/R_1$. For $k > 1$ we can write  
\begin{align} \label{eq:updateqr1}
AV_{k} = [AV_{k-1}, Av_{k-1}] = [Q_{k-1} R_{k-1}, Av_{k-1}] 
=\underbrace{[Q_{k-1}, q_{k}]}_{:=Q_{k}} \underbrace{\begin{pmatrix} R_{k-1} & r_{k} \\ 0 & r_{k,k} \end{pmatrix}}_{:=R_{k}}
\end{align} 
\begin{align} \label{eq:updateqrbis}
\text{with} \hspace{0.2cm} r_{k} = Q_{k - 1}^T Av_{k-1},\hspace{0.2cm} \tilde{q}_{k} = Av_{k - 1} - Q_{k-1} r_{k}, \hspace{0.2cm} r_{k,k} = || \tilde{q}_{k} ||  \hspace{0.2cm}\text{and}\hspace{0.2cm}  q_{k} = \tilde{q}_{k}/r_{k,k}.
\end{align}
When $r_{k,k}=0$ we can simply replace $\tilde{q}_{k}$ with any vector that is orthogonal to $Q_{k-1}$ \cite{LAMPE20122845,daniel1976reorthogonalization}. 
The matrix $R_{k}^T R_{k}$ can also be efficiently computed by
\begin{equation} \label{eq:rtr}
R_{k}^T R_{k} = \begin{pmatrix} R^T_{k-1} & 0 \\  r^T_{k} & r_{k,k} \end{pmatrix} \begin{pmatrix} R_{k-1} & r_{k} \\ 0 & r_{k,k} \end{pmatrix} 
= \begin{pmatrix} R^T_{k-1}R_{k-1} & R_{k-1}^T r_{k} \\  r^T_{k} R_{k-1} &r^T_{k}r_{k} + r_{k,k}^2 \end{pmatrix}. 
\end{equation}

Up to this point, we have not used any additional structure of the twice continuously differentiable convex function $\Psi(x)$. In general, calculating the gradient $\nabla \Psi(x)$ or the Hessian $\nabla^2 \Psi(x)$  could be quite computationally expensive. For instance, if we consider $\Psi(x) = \Psi_p(Lx)$ the smooth approximation of $\frac{1}{p}||Lx||_p^{p}$, evaluating this gradient requires a matrix-vector product with both $L$ and $L^T$. Let us consider $\Psi(x) = \tilde{\Psi}(Lx)$ with a matrix $L\in\mathbb{R}^{s \times n}$ and twice continuously differentiable convex function $\tilde{\Psi} :\mathbb{R}^s \longrightarrow \mathbb{R}$. Recall that for the $\ell_p$ norm we take $\tilde{\Psi}(z) = \Psi_p(z)$, while for the general form Tikhonov problem we have $\tilde{\Psi}(z) = \frac{1}{2}||z||_2^2$. Further improvements in the latter case will be presented in \cref{sec:genform}.

In addition to saving the tall and skinny matrices $V_k$ and $AV_k$ we also save the matrix $LV_k \in \mathbb{R}^{s \times k}$.
We introduce recurrences for $z_k := Lx_k$ and $t_k :=Ax_k$, i.e. we get
\begin{align*}
z_k =&~ Lx_k = Lx_{k-1} + \gamma_k L \Delta x_k =  z_{k-1} + \gamma_k LV_k\Delta y_k=  z_{k-1} + \gamma_k \Delta z_k \\
t_k =&~ Ax_k = Ax_{k-1} +\gamma_k A\Delta x_k = t_{k-1} +  \gamma_k AV_k\Delta y_k= t_{k-1} +  \gamma_k \Delta t_k
\end{align*}
with $\Delta z_k = (LV_k)\Delta y_k$ and  $\Delta t_k = (AV_{k})\Delta y_{k}$ computed as tall and skinny matrix vector products and $z_0=t_0 = 0$.
When we consider the case $L=I_n$ we obviously get the simplification $z_k  = x_k$. The Hessian $V_k^T \nabla^2 \Psi(x_{k-1}) V_k$ in the projected Jacobian \cref{eq:jacobian} with $y=\bar{y}_{k-1}$ can be rewritten as
\begin{equation*}
V_k^T \nabla^2 \Psi(x_{k-1}) V_k = (LV_k)^T  \nabla^2 \tilde{\Psi}(z_{k-1}) (LV_k).
\end{equation*}
Remember that for our example $\tilde{\Psi}(z) = \Psi_p(z)$ the matrix $ \nabla^2  \tilde{\Psi}(z_{k-1})$ is diagonal.
To summarize, the linear system \cref{eq:newtoneq} to calculate the Projected Newton direction can be rewritten as 
\begin{align}\label{eq:solve} 
\begin{pmatrix}\lambda_{k-1}R_k^T R_k + (LV_k)^T  \nabla^2 \tilde{\Psi}(z_{k-1}) (LV_k) & R_k^T R_k \bar{y}_{k-1} - d_k \\  \nonumber 
\left(R_k^T R_k \bar{y}_{k-1} - d_k\right)^T & 0  
 \end{pmatrix} \begin{pmatrix} \Delta y_{k}\\ \Delta \lambda_{k} \end{pmatrix} = \hspace{1.5cm} &  \\[2ex]
- \begin{pmatrix}\lambda_{k-1}\left(R_k^T R_k \bar{y}_{k-1} - d_k\right) + (LV_k)^T \nabla \tilde{\Psi}(z_{k-1}) \\ \frac{1}{2}||t_{k-1} - b||^2 - \frac{\sigma^2}{2}& \end{pmatrix}&.
\end{align}

\subsection{Efficiently performing the backtracking line search}\label{sec:improvementsback}

Similarly as before, we now also save the tall and skinny matrix $A^T AV_{k}\in\mathbb{R}^{n\times k}$ and introduce a recurrence relation for $w_k = A^T A x_k $. More specifically, we have
\begin{align*}
w_k =&~ A^T A x_k = A^T A x_{k-1} + \gamma_k A^T A \Delta x_{k} =  w_{k-1} +  \gamma_k \Delta w_k
\end{align*}
where we compute $\Delta w_k = (A^T AV_{k})\Delta y_{k}$ using a tall and skinny matrix-vector product and initialize $w_0 = 0$. Note that we need to perform only a tall and skinny matrix-vector product with $LV_k,AV_k$ and $A^T A V_k$ and that we then can compute $z_k,t_k$ and $w_k$ for many different values of the step-length $\gamma_k$ with only vector additions. This is very useful for efficiently performing the backtracking line search.
We calculate the gradient in \cref{eq:bigf} using only a matrix-vector product with $L^T$ (no matrix-vector product with $L$ anymore), i.e. we have: 
\begin{equation*}
\nabla \Psi (x_{k}) = L^T \nabla \tilde{\Psi}(z_k).
\end{equation*}
The above definitions now allow us to efficiently perform the backtracking line search. Indeed,  using the definition for $t_k, w_k$ and $z_k$ we can write
\begin{equation} \label{eq:calcf2}
F(x_k,\lambda_k) = \begin{pmatrix} \lambda_k \left(w_k - A^T b \right) + L^T \nabla \tilde{\Psi}(z_k)\\
\frac{1}{2}||t_k - b||^2 - \frac{\sigma^2}{2}  \end{pmatrix}.
\end{equation}
Hence, no additional matrix-vector products with $A, A^T$ or $L$ are needed to compute this vector, only one matrix-vector product with $L^T$ each time the step-length is reduced. Note that $\tilde{v}_k$, as defined in \cref{eq:vtilde}, is the first component of this function, so we do not need to perform any additional calculations to obtain this vector. Using the above, we can now formulate the Projected Newton method for solving \cref{eq:reform} with $\Psi(x) = \tilde{\Psi}(Lx)$, see \cref{alg:PNTM}. 

\begin{algorithm}
\caption{Projected Newton method for $\Psi(x) = \tilde{\Psi}(Lx)$}
\label{alg:PNTM} 
\begin{algorithmic}[1] \small
\STATE{ $\bar{y}_0 = 0$; $z_0 = 0$; $t_0 = 0$; $w_0 = 0$; $\tau = 0.9$; $c = 10^{-4}$;\hfill $\#$ Initializations}
\STATE{ $AV_0 =  A^T A V_0 = LV_0 = \emptyset;$}
\STATE{ $\tilde{v}_0 = A^Tb$; $d_1 = ||\tilde{v}_0 ||;$ $v_0 =\tilde{v}_0 /d_1$; $V_1 = v_0$;   \hfill $\#$ Matvec with $A^T$}
\STATE{ $F_0 = \left(-\lambda_0 \tilde{v}^T_0, \frac{1}{2}||b||^2 - \frac{\sigma^2}{2}\right)^T; $\hfill $\#$ Compute $F_0 = F(x_0,\lambda_0)$} 
\FOR{$k = 1,2,\ldots,\text{until convergence}$} 
\STATE{$AV_{k} = [AV_{k-1}, Av_{k-1}]$; $A^TAV_{k} = [A^TAV_{k-1}, A^T(Av_{k-1})]$; \label{matvec}\hfill $\#$ Matvec with $A$ and $A^T$}
\STATE{$LV_{k} = [LV_{k-1}, Lv_{k-1}]$; \hfill $\#$ Matvec with $L$}
\STATE{Calculate $Q_{k},R_{k}$ and $R_{k}^T R_{k}$ by \cref{eq:updateqr1}-\cref{eq:rtr} \label{lineqr}\hfill $\#$ QR decomposition $AV_k$}
\STATE{ Calculate $\Delta y_k$ and $\Delta \lambda_k$ by solving \cref{eq:solve} \label{linesolve} \hfill  $\#$ Projected Newton direction}
\STATE{$\lambda_{k} = \lambda_{k-1} + \Delta \lambda_{k};$}  
\STATE{ \textbf{if} $\lambda_k > 0$ \textbf{ then } $\gamma_k=1;$ \textbf{ else } $\gamma_k = -\tau\lambda_{k-1} / \Delta \lambda_{k}; \lambda_{k} = \lambda_{k-1} + \gamma_k \Delta \lambda_{k};$  \textbf{ end if } \label{poslam}} 
\STATE{ $\Delta z_k = (LV_k) \Delta y_k;$ $\Delta t_k = (AV_k) \Delta y_k;$ $\Delta w_k = (A^TAV_k) \Delta y_k;$ \label{tallskinny}}
\STATE{$(y^T_{k},z^T_k,t^T_k,w^T_k) = (\bar{y}^T_{k-1},z^T_{k-1},t^T_{k-1},w^T_{k-1}) + \gamma_k (\Delta y^T_{k}, \Delta z^T_k,\Delta t^T_k,\Delta w^T_k);$}
\STATE{$\tilde{v}_{k} = \lambda_k\left(w_k - \tilde{v}_0 \right) + L^T \nabla \tilde{\Psi} (z_{k});$ \label{vk1} \hfill $\#$ Matvec with $L^T$  }
\STATE{ $F_k = \left(\tilde{v}_k^T,\frac{1}{2}||t_k - b||^2 - \frac{\sigma^2}{2}\right)^T;$ \hfill $\#$ Compute $F_k = F(x_k,\lambda_k)$}
\STATE {\textbf{while }$\frac{1}{2}||F_k||^2 \geq (\frac{1}{2} - c\gamma_k) || F_{k-1}|| ^2$  \label{lineback} \textbf{ do } \hfill $\#$Backtracking line search}
\STATE{\hspace{0.1cm} $\gamma_k = \tau \gamma_{k}$;}
\STATE{\hspace{0.1cm} $(y^T_{k},z^T_k,t^T_k,w^T_k,\lambda_k) = (\bar{y}^T_{k-1},z^T_{k-1},t^T_{k-1},w^T_{k-1},\lambda_{k-1}) + \gamma_k (\Delta y^T_{k}, \Delta z^T_k,\Delta t^T_k,\Delta w^T_k,\Delta \lambda_k);$}
\STATE{\hspace{0.1cm} $\tilde{v}_{k} = \lambda_k\left(w_k - \tilde{v}_0 \right) + L^T \nabla \tilde{\Psi} (z_{k});$ \label{vk2} \hfill $\#$ Matvec with $L^T$  }
\STATE{ \hspace{0.1cm} $F_k = \left(\tilde{v}_k^T,\frac{1}{2}||t_k - b||^2 - \frac{\sigma^2}{2}\right)^T;$ \hfill $\#$ Compute $F_k = F(x_k,\lambda_k)$}
\STATE{ \textbf{end while}}
\STATE{$\bar{y}_{k} = (y_{k}^T,0)^T;$ $d_{k+1}= (d_k^T,0)^T;$}
\STATE{$v_{k} = \tilde{v}_{k} - \sum_{j=0}^{k-1} (v_{j}^T \tilde{v}_{k}) v_{j};$ \label{gsorth}\hfill $\#$ Gram-Schmidt}
\STATE{$v_{k} = v_k/||v_k||;$ $V_{k + 1} = [V_k,v_{k}]$ \hfill $\#$ Normalize and add new vector to basis}
\ENDFOR
\STATE{$x_k = V_k y_k;$}
\end{algorithmic}
\end{algorithm}

\subsection{Computational cost of the Projected Newton method} \label{sec:flops}
Let us comment on the computational cost of the Projected Newton method as implemented in \cref{alg:PNTM}. Suppose we are at iteration $k$ in the algorithm. First, on line \ref{matvec}, we need to perform a matrix vector product with $A$ and $A^T$. The standard multiplication of a dense $m \times n$ matrix with a vector takes $2mn - m$ floating point operations (\textsc{flop}s). However, we are mainly interested in applications with a sparse matrix $A$. The number of \textsc{flop}s for a sparse matrix-vector multiplication with $A$ is much lower, namely $\mathcal{O}(\mu_A n)$ where $\mu_A$ is the average number of non-zero entries per row of $A$. Let us denote $\textsc{flop}s(A)$ the number of \textsc{flop}s required for a matrix vector product with $A$ and likewise we denote $\textsc{flop}s(A^T),\textsc{flop}s(L)$ and $\textsc{flop}s(L^T)$ the number \textsc{flop}s required for a matrix vector product $A^T,L$ and $L^T$ respectively. The total number of \textsc{flop}s required for these matrix vector products in iteration $k$ is then given by
\begin{equation}
\textsc{flop}s(A) + \textsc{flop}s(A^T) + \textsc{flop}s(L) + (\# \text{red}_k)\textsc{flop}s(L^T),
\end{equation}
where $\#\text{red}_k$ denotes the number of times the step-length is reduced in the backtracking line search (see line \ref{lineback}). 

Next, let us consider the computational cost of updating the QR decomposition of $AV_k$ by \cref{eq:updateqr1}-\cref{eq:rtr}. Calculating the vector $r_{k}$ in \cref{eq:updateqrbis} can be seen as a dense matrix vector product with $Q_{k-1}^T\in\mathbb{R}^{(k-1) \times m}$ and hence requires $2(k-1)m - (k-1)\approx 2mk$ \textsc{flop}s. The  number of \textsc{flop}s for calculating $\tilde{q}_k$ is also approximately $2mk$ and thus we have that the total number of \textsc{flop}s for updating the QR decomposition on line \ref{lineqr} is approximately $4mk$. Here and it what follows we disregard all terms that do not have a factor $m,n$ or $s$. In addition, we also disregard all terms that do not depend on the iteration index $k$, with the obvious exception of the matrix vector products with $A,A^T,L$ and $L^T$. 

The cost of forming the right hand side in \cref{eq:solve} is dominated by the tall and skinny matrix vector product $(LV_k)^T \nabla \tilde{\Psi}(z_{k-1})$ and takes approximately $2sk$ \textsc{flop}s. The cost of forming the projected Jacobian in \cref{eq:solve} is dominated by the construction of $(LV_k)^T  \nabla^2 \tilde{\Psi}(z_{k-1}) (LV_k)$. Multiplication of the diagonal matrix $\nabla^2 \tilde{\Psi}(z_{k-1})$ with $LV_k$ requires $sk$ \textsc{flop}s. Due to the symmetry of the matrix $(LV_k)^T  \nabla^2 \tilde{\Psi}(z_{k-1}) (LV_k)$ we only need to compute $k(k+1)/2$ entries, which all can be computed as an inner-product between two vectors of length $s$. Recall that such an inner product takes $2s-1$ \textsc{flops}. This which means that the total cost of forming this matrix is $k(k+1)(2s-1)/2 + sk \approx sk^2 + 2sk$ \textsc{flop}s.

On line \ref{tallskinny} we compute tall and skinny matrix vector products with $LV_k\in\mathbb{R}^{s\times k}$,$AV_k\in\mathbb{R}^{m\times k}$ and $A^TAV_k\in\mathbb{R}^{n\times k}$, which approximately requires $2sk + 2mk + 2nk$ \textsc{flop}s. Lastly, we consider the orthogonalization step on line \ref{gsorth}. Calculating the $k$ inner products with vectors of length $n$ takes $(2n-1)k$ \textsc{flops} and the vector scaling and subtractions take another $2nk$ \textsc{flops}. Hence, the orthogonalization step approximately requires $4nk$ \textsc{flop}s in total. To summarize, the number of \textsc{flops} required in iteration $k$ of \cref{alg:PNTM} is approximately equal to
\begin{equation}\label{eq:flops}
\textsc{flop}s(A) + \textsc{flop}s(A^T) + \textsc{flop}s(L) + (\# \text{red}_k)\textsc{flop}s(L^T) + 6(s + n + m)k + sk^2. 
\end{equation}
From this expression it is clear that the algorithm should terminate after a relatively small amount of iterations $k$ if we want to get good performance, since otherwise the term $sk^2$ will become very large. In \cref{sec:stopping} we discuss possible stopping criteria and thresholds that can be used in practice to obtain good performance. 

For large sparse matrices $A$ and $L$, we have that the dominant cost of \cref{alg:PNTM} in the first few iterations is given by the number of matrix vector product with $A$, $A^T$, $L$ and $L^T$. However, after a few iterations, the term $sk^2$ will quickly start to dominate the computational cost. When precisely this term starts to dominate obviously depends on the sparsity of the matrices $A$ and $L$ and the dimensions $m,n$ and $s$. The sparser the matrices, the sooner the term $sk^2$ starts to dominate. For large dense matrices we have that the total number of \textsc{flop}s remains dominated by the cost of the matrix vector products as long as $k^2\ll \min\{m,n,s\}.$ 

\subsection{Nonsingularity of the projected Jacobian} \label{sec:nonsing}
In this section we describe a strategy that guarantees that the projected Jacobian \cref{eq:jacobian} in \cref{alg:PNTM} is nonsingular for all iterations $k$.  
Note that line \ref{poslam} has been added in \cref{alg:PNTM} to make sure the regularization parameter $\lambda_k$ remains strictly positive. First of all we know that the Lagrange multiplier $\lambda^*$ corresponding to the solution $x^{*}$ of \cref{eq:reform} has to be strictly positive so it is natural to require that the iterates also remain strictly positive. More importantly, the following lemma holds for a strictly positive Lagrange multiplier $\lambda_k$.
\begin{lemma} \label{thm:nonsingularity}
Let $\lambda_k > 0$. The projected Jacobian $J^{(k + 1)}(\bar{y}_{k},\lambda_k)$ is nonsingular if and only if 
\begin{equation}
V_{k+1}^T A^T (AV_{k+1}\bar{y}_k - b) \neq 0. 
\end{equation} 
\end{lemma}
\begin{proof}
Let us denote $M_{k+1} = \lambda_k V_{k+1}^T A^T A V_{k+1} + V_{k+1}^T\nabla^2 \Psi(V_{k+1}\bar{y}_k) V_{k+1}$. 
Due to the $2\times 2$ block structure of the projected Jacobian
\begin{equation*}
J^{(k+1)}(\bar{y}_k,\lambda_k)  = \begin{pmatrix} M_{k+1} & V_{k+1}^T A^T (AV_{k+1}\bar{y}_{k+1} - b) \\  (AV_{k+1}\bar{y}_{k+1} - b)^TAV_{k+1} & 0 \end{pmatrix}
\end{equation*}
we have the following expression for the determinant \cite{zhang2006schur}:
\begin{equation*}
\det J^{(k+1)}(\bar{y}_k,\lambda_k) = - (AV_{k+1}\bar{y}_{k} - b)^TAV_{k+1}M_{k+1}^{-1}V_{k+1}^T A^T (AV_{k+1}\bar{y}_{k} - b) \det M_{k+1}. 
\end{equation*}
Since $\lambda_k>0$ we know that the matrix $M_{k+1}$ is positive definite and as a consequence we have that $\det M_{k+1} \neq 0$. Moreover, since $M_{k+1}^{-1}$ is also positive definite we know that 
\begin{equation*}
 (AV_{k+1}\bar{y}_{k} - b)^TAV_{k+1}M_{k+1}^{-1}V_{k+1}^T A^T (AV_{k+1}\bar{y}_{k} - b) = 0 \Leftrightarrow V_{k+1}A^T (AV_{k+1}\bar{y}_{k} - b) = 0. 
\end{equation*}
This concludes the proof. 
\end{proof}

Using the above lemma we can show that we can always find a suitable step-length such that the Jacobian matrices in \cref{alg:PNTM} remain nonsingular. 
\begin{lemma}
Suppose that the projected Jacobian $J^{(k)}(\bar{y}_{k-1},\lambda_{k-1})$ in iteration $k$ is nonsingular and $\lambda_{k-1}>0$. Then we can always find a step-length $\gamma_k$ such that the sufficient decrease condition \cref{eq:sufficient_decrease} is satisfied and such that $J^{(k+1)}(\bar{y}_k,\lambda_k)$ is nonsingular. 
\end{lemma}
\begin{proof}
We know by \cref{thm:descent} that there exists a step-length $\gamma^{(1)}>0$ such that the sufficient decrease condition \cref{eq:sufficient_decrease} is satisfied for all $\gamma \in ]0,\gamma^{(1)}]$. Since $\lambda_{k-1} > 0$ there also exists a step-length $\gamma^{(2)}>0$ such that $\lambda_{k} = \lambda_{k-1} + \gamma \Delta \lambda_k>0$ for all $\gamma \in ]0,\gamma^{(2)}]$. 
Now let $\bar{y}_{k} = (y^T_{k},0)^T$ with $y_{k} = \bar{y}_{k-1} + \gamma \Delta y_k$. The first $k$ elements of the vector 
\begin{equation}
V_{k+1}A^T (AV_{k+1}\bar{y}_{k} - b) = R_{k+1}^T R_{k+1}\bar{y}_k - d_{k+1}.
\end{equation}
are given by the vector $R_{k}^T R_{k} y_k - d_k = R_{k}^T R_{k} (\bar{y}_{k-1} + \gamma \Delta y_k) - d_k$, see \cref{eq:rtr}. Since  $J^{(k)}(\bar{y}_{k-1},\lambda_{k-1})$ is nonsingular it follows from \cref{thm:nonsingularity} that $R_{k}^T R_{k}\bar{y}_{k-1} - d_k\neq 0$ and as a consequence there exist a $\gamma^{(3)}$ such that $V_{k+1}A^T (AV_{k+1}\bar{y}_{k} - b) \neq 0$ for all   $\gamma \in ]0,\gamma^{(3)}]$. Now it suffices to take $\gamma_k \in]0,\min\{\gamma^{(1)},\gamma^{(2)},\gamma^{(3)}\}]$.
\end{proof}
\begin{remark}
We could easily modify the backtracking line-search on line \ref{lineback} of \cref{alg:PNTM} to ensure that $||R_{k}^T R_{k} y_k - d_k||>0$, which in its turn would ensure that the Jacobian $J^{(k+1)}(\bar{y}_k,\lambda_k)$ is nonsingular. However, in practice, we have never encountered a case where this value becomes small, so we do not include it in the description of the algorithm. 
\end{remark}

\subsection{Further improvements for general form Tikhonov regularization} \label{sec:genform}
If we choose $\Psi(x) = \frac{1}{2}||Lx||^2_2$ with $L\in\mathbb{R}^{s\times n}$ (or equivalently $\tilde{\Psi}(z) = \frac{1}{2}||z||^2$) in \cref{eq:reform} we recover the general form Tikhonov problem
\begin{equation} \label{eq:reform2}
\min_{x\in\mathbb{R}^n} \hspace{0.2cm} \frac{1}{2}||Lx||^{2}_2 \hspace{1cm} \text{subject to} \hspace{1cm} \frac{1}{2} ||Ax - b||^{2}_2 = \frac{\sigma^2}{2}
\end{equation}
and \cref{alg:PNTM} can be improved even further. First we observe that $\nabla \Psi(x) = L^T L x$ and $\nabla^2 \Psi(x) = L^T L$ (or equivalently $\nabla \tilde{\Psi}(z) = z$ and $\nabla^2 \tilde{\Psi}(z) = I_s$). By considering a reduced QR decomposition of the tall and skinny matrix $LV_k \in \mathbb{R}^{s\times k}$ and also saving the matrix $L^T LV_k \in \mathbb{R}^{n\times k}$ we can reorganize the algorithm in such a way that we only need a single matrix-vector product with $L^T$ each iteration, instead of one for each time we need to compute $F(x_k,\lambda_k)$ for the backtracking line search. This leads to an improvement when the step-length $\gamma_k$ has to be reduced  multiple times before the sufficient decrease condition is satisfied. Moreover, the QR decomposition of $LV_k$ also allows us to remove the term $sk^2$ in \cref{eq:flops} due to the construction of the projected Jacobian, which gives us a significant improvement over the implementation in \cref{alg:PNTM}.

Let $\tilde{Q}_k \in \mathbb{R}^{s \times k}$ with orthonormal columns and $\tilde{R}_k \in\mathbb{R}^{k\times k}$ upper-triangular, such that the reduced QR decomposition of $LV_k$ is given by $LV_k = \tilde{Q}_k \tilde{R}_k.$ The Hessian for the regularization term in \cref{eq:solve} can be rewritten as
\begin{equation*}
 (LV_k)^T  \nabla^2 \tilde{\Psi}(z_{k-1}) (LV_k) = (LV_k)^T (LV_k) = \tilde{R}_k^T \tilde{R}_k.
\end{equation*}
Similarly we have for the gradient in the right-hand side of \cref{eq:solve} the following simplification
\begin{equation*}
(LV_k)^T \nabla \tilde{\Psi}(z_{k-1}) = (LV_k)^T L x_{k-1} = (LV_k)^T L V_k \bar{y}_{k-1} = \tilde{R}_k^T \tilde{R}_k \bar{y}_{k-1}.
\end{equation*}
When the basis $V_k$ is expanded, we can efficiently update the QR decomposition of $LV_k$, similarly as in \cref{eq:updateqr1} - \cref{eq:rtr}.
Next, let us define a new auxiliary vector $u_k = L^T L x_k$. We again have a recurrence for this variable:
\begin{equation*}
u_k =~ L^T L x_k = L^T L x_{k-1} + \gamma_k L^T L \Delta x_{k} = u_{k-1} + \gamma_k L^T L V_{k} \Delta y_{k} =  u_{k-1} +  \gamma_k \Delta u_k
\end{equation*}
where we compute $\Delta u_k = (L^T L V_{k}) \Delta y_{k}$ as a tall and skinny matrix-vector product and initialize $u_0 = 0$. Now we can remove the matrix-vector product with $L^T$ from the backtracking line search. Indeed, we simply replace the term $L^T \nabla \tilde{\Psi}(z_k)$ with $u_k$, since $\nabla \tilde{\Psi}(z_k) = z_k = Lx_k$. Note that due to these simplifications, there is no need for the auxiliary variable $z_k$ anymore. Hence, we also do not need the tall and skinny matrix-vector product $\Delta z_k = (LV_k) \Delta y_k$ on line $\ref{tallskinny}$ in \cref{alg:PNTM}. 
To summarize, we construct and save the matrix $L^T L V_{k}$, which requires one matrix-vector product with $L^T$. We remove the auxiliary variable $z_k$ and add the new variable $u_k$. We replace the tall and skinny matrix-vector product for $\Delta z_k$  with $\Delta u_k = (L^T L V_{k}) \Delta y_{k}$ on line \ref{tallskinny}. We replace the construction of $\tilde{v}_k$  on lines \ref{vk1} and \ref{vk2} with 
\begin{equation*}
\tilde{v}_k =  \lambda_k\left(w_k - \tilde{v}_0 \right) + u_k
\end{equation*}
such that there is no need anymore to compute it using a matrix-vector product with $L^T$. Lastly, we compute the reduced QR decomposition of $LV_k$ and instead of calculating the Projected Newton direction on line \ref{linesolve} using \cref{eq:solve} we now calculate it as
\begin{align}\label{eq:solve_gen}
&\begin{pmatrix}\lambda_{k-1}R_k^T R_k + \tilde{R}_k^T \tilde{R}_k& R_k^T R_k \bar{y}_{k-1} - d_k \\ 
\left(R_k^T R_k \bar{y}_{k-1} - d_k\right)^T & 0  
 \end{pmatrix} \begin{pmatrix} \Delta y_{k}\\ \Delta \lambda_{k} \end{pmatrix} =  \nonumber \\
 & \hspace{5cm}-\begin{pmatrix}\lambda_{k-1}\left(R_k^T R_k \bar{y}_{k-1} - d_k\right) + \tilde{R}_k^T \tilde{R}_k \bar{y}_{k-1}\\ \frac{1}{2}||t_k - b||^2 - \frac{\sigma^2}{2}& \end{pmatrix}.
\end{align}

See \cref{alg:PNTM_gen} for a detailed description of the method. The most important difference with the (more general) implementation given by \cref{alg:PNTM} is that there is no more matrix-vector product with $L^T$ present in the backtracking line search and that the construction of the matrix  $(LV_k)^T  \nabla^2 \tilde{\Psi}(z_{k-1}) (LV_k)$ in \cref{eq:solve} has been replaced by $\tilde{R}_k^T\tilde{R}_k$. Recall that the term $sk^2$ in \cref{eq:flops} was due to the construction of this matrix. Updating the QR decomposition of $LV_k$ is much cheaper, only approximately $4sk$ \textsc{flop}s, so this is a significant improvement.  

\begin{algorithm}
\caption{Projected Newton method for general form Tikhonov regularization}
\label{alg:PNTM_gen} 
\begin{algorithmic}[1] \small
\STATE{ $\bar{y}_0 = 0$; $u_0 = 0$; $t_0 = 0$; $w_0 = 0$; $\tau = 0.9$; $c = 10^{-4}$;\hfill $\#$ Initializations}
\STATE{ $AV_0 =  A^T A V_0 = LV_0 = L^T L V_0 = \emptyset;$}
\STATE{ $\tilde{v}_0 = A^Tb$; $d_1 = ||\tilde{v}_0 ||;$ $v_0 =\tilde{v}_0 /d_1$; $V_1 = v_0$;   \hfill $\#$ Matvec with $A^T$}
\STATE{ $F_0 = \left(-\lambda_0 \tilde{v}^T_0, \frac{1}{2}||b||^2 - \frac{\sigma^2}{2}\right)^T; $\hfill $\#$ Compute $F_0 = F(x_0,\lambda_0)$} 
\FOR{$k = 1,2,\ldots,\text{until convergence}$} 
\STATE{$AV_{k} = [AV_{k-1}, Av_{k-1}]$; $A^TAV_{k} = [A^TAV_{k-1}, A^T(Av_{k-1})]$; \hfill $\#$ Matvec with $A$ and $A^T$}
\STATE{$LV_{k} = [LV_{k-1}, Lv_{k-1}]$; $L^TLV_{k} = [L^TLV_{k-1}, L^T(Lv_{k-1})]$; \hfill $\#$ Matvec with $L$ and $L^T$}
\STATE{Calculate $Q_{k},R_{k}$ and $R_{k}^T R_{k}$ by \cref{eq:updateqr1}-\cref{eq:rtr} \hfill $\#$ QR decomposition $AV_k$}
\STATE{Calculate $\tilde{Q}_{k},\tilde{R}_{k}$ and $\tilde{R}_{k}^T \tilde{R}_{k}$ (similarly) \hfill $\#$ QR decomposition $LV_k$}
\STATE{ Calculate $\Delta y_k$ and $\Delta \lambda_k$ by solving \cref{eq:solve_gen}\hfill  $\#$ Projected Newton direction}
\STATE{$\lambda_{k} = \lambda_{k-1} + \Delta \lambda_{k};$}  
\STATE{ \textbf{if} $\lambda_k > 0$ \textbf{ then } $\gamma_k=1;$ \textbf{ else } $\gamma_k = -\tau\lambda_{k-1} / \Delta \lambda_{k}; \lambda_{k} = \lambda_{k-1} + \gamma_k \Delta \lambda_{k};$  \textbf{ end if }} 
\STATE{ $\Delta u_k = (L^TLV_k) \Delta y_k;$ $\Delta t_k = (AV_k) \Delta y_k;$ $\Delta w_k = (A^TAV_k) \Delta y_k;$ \label{tallskinny2}}
\STATE{$(y^T_{k},u^T_k,t^T_k,w^T_k) = (\bar{y}^T_{k-1},u^T_{k-1},t^T_{k-1},w^T_{k-1}) + \gamma_k (\Delta y^T_{k}, \Delta u^T_k,\Delta t^T_k,\Delta w^T_k);$}
\STATE{$\tilde{v}_{k} = \lambda_k\left(w_k - \tilde{v}_0 \right) + u_k$;}
\STATE{$F_k = \left(\tilde{v}_k^T,\frac{1}{2}||t_k - b||^2 - \frac{\sigma^2}{2}\right)^T;$ \hfill $\#$ Compute $F_k = F(x_k,\lambda_k)$}
\STATE {\textbf{while }$\frac{1}{2}||F_k||^2 \geq (\frac{1}{2} - c\gamma_k) || F_{k-1}|| ^2$ \textbf{ do } \hfill $\#$Backtracking line search}
\STATE{\hspace{0.1cm} $\gamma_k = \tau \gamma_{k}$;}
\STATE{\hspace{0.1cm} $(y^T_{k},u^T_k,t^T_k,w^T_k,\lambda_k) = (\bar{y}^T_{k-1},u^T_{k-1},t^T_{k-1},w^T_{k-1},\lambda_{k-1}) + \gamma_k (\Delta y^T_{k}, \Delta u^T_k,\Delta t^T_k,\Delta w^T_k,\Delta \lambda_k);$}
\STATE{\hspace{0.1cm} $\tilde{v}_{k} = \lambda_k\left(w_k - \tilde{v}_0 \right) + u_k;$}
\STATE{ \hspace{0.1cm} $F_k = \left(\tilde{v}_k^T,\frac{1}{2}||t_k - b||^2 - \frac{\sigma^2}{2}\right)^T;$ \hfill $\#$ Compute $F_k = F(x_k,\lambda_k)$}
\STATE{ \textbf{end while}}
\STATE{$\bar{y}_{k} = (y_{k}^T,0)^T;$$d_{k+1}= (d_k^T,0)^T;$}
\STATE{$v_{k} = \tilde{v}_{k} - \sum_{j=0}^{k-1} (v_{j}^T \tilde{v}_{k}) v_{j};$ \hfill \label{orth2} $\#$ Gram-Schmidt}
\STATE{$v_{k} = v_k/||v_k||;$ $V_{k + 1} = [V_k,v_{k}]$ \hfill $\#$ Normalize and add new vector to basis}
\ENDFOR
\STATE{$x_k = V_k y_k;$}
\end{algorithmic}
\end{algorithm}

Similarly as in \cref{sec:flops}, it is readily checked that the number of \textsc{flop}s required in iteration $k$ of algorithm \cref{alg:PNTM_gen} is given by
\begin{equation}\label{eq:flops2}
\textsc{flop}s(A) + \textsc{flop}s(A^T) + \textsc{flop}s(L) + \textsc{flop}s(L^T) + 6mk + 4sk + 8nk,  
\end{equation}
where we again disregard terms that no not have a factor $m,n$ or $s$ or do not depend on the iteration index $k$, with the exception of the matrix vector products with $A$, $A^T$, $L$ and $L^T$. Indeed, the QR decompositions of $AV_k$ and $LV_k$ require about $4mk$ and $4sk$ \textsc{flop}s respectively. The tall and skinny matrix-vector products on line \ref{tallskinny2} in \cref{alg:PNTM_gen} take about $4nk + 2mk$ \textsc{flop}s and the orthogonalization step on line \ref{orth2} still requires about $4nk$ \textsc{flop}s. 

\begin{remark}
It is possible to replace the $\ell_2$ norm for the data fidelity term with the $\ell_q$ norm with $1\leq q \leq 2$, i.e. to consider the constrained optimization problem
\begin{equation}\label{eq:noise_lplq}
\min_{x\in\mathbb{R}^n} \hspace{0.2cm} \frac{1}{p}||Lx||^p_p \hspace{1cm} \text{subject to} \hspace{1cm} \frac{1}{q} ||A x - b||^{q}_q = \frac{1}{q}||e||_q^q.
\end{equation}
It is straightforward to extend the Projected Newton method such that it can solve (a smooth approximation of) this problem. We can simply take the smooth approximation $\Psi_q(Ax - b)$ of $\frac{1}{q}||Ax - b||_q^q$ and consider computational improvements similarly to how we treated $\Psi_{p}(Lx)$. 
\end{remark}

\section{Reference methods}\label{sec:refmeth}

This section describes two reference methods that we use to compare the Projected Newton method with in \cref{sec:num}. The first one can be used to solve the general form Tikhonov problem \cref{eq:reform2}, while the second one can be used to obtain an approximate solution for the $\ell_p$ regularized problem \cref{eq:lp} with a fixed regularization parameter $\alpha$. 

\subsection{Generalized Krylov subspace method for general form Tikhonov regularization} \label{sec:gkssec}
The \textit{Generalized Krylov subspace} (GKS) Tikhonov regularization method developed in \cite{LAMPE20122845}, see \cref{alg:GKS}, can only be used to solve the general form Tikhonov problem.
However, since it serves as inspiration for this work and bears resemblance to our approach, we believe it deserves some attention. It also constructs a Generalized Krylov subspace basis $V_{k} \in \mathbb{R}^{n \times k}$ for $k\geq l$ (although a different one) and computes the solution to the projected problem 
\begin{equation}\label{eq:ygks}
y_k  = \argmin_{y\in\mathbb{R}^{k}} ||AV_{k}y - b||_2^2 + \alpha_k ||LV_{k}y||_2^2.
\end{equation}
where $\alpha_k$ is determined such that $||AV_{k}y_k - b|| = \sigma$. Note that $\alpha_k$ can be determined using a scalar root-finder, which requires that \cref{eq:ygks} is solved multiple times. The GKS algorithm starts for $k=l$ with some initial $l$-dimensional orthonormal basis $V_l \in\mathbb{R}^{n\times l}$, for instance the basis for the Krylov subspace $\mathcal{K}_l(A^TA,A^Tb)$, where the dimension $l$ of the initial basis is large enough such that a regularization parameter $\alpha_l$ exists that satisfies the discrepancy principle $||AV_l y_l - b|| = \sigma$.
Subsequently, the basis is expanded by adding the normalized residual of the unreduced problem, i.e. 
\begin{equation}\label{eq:vnew}
\tilde{v} = (A^T A + \alpha_k L^T L)V_{k} y_k - A^Tb, \hspace{1cm}v_{new} = \tilde{v}/||\tilde{v}||. 
\end{equation}

In exact arithmetic this vector is already orthogonal to $V_{k}$. However, the authors in \cite{LAMPE20122845} suggest to add a re-orthogonalization step to enforce orthogonality in the presence of round-off errors. 

An efficient implementation of \cref{alg:GKS} updates the QR decomposition of the matrices $AV_k$ and $LV_k$ similarly as in \cref{alg:PNTM_gen}. Moreover, each iteration also requires a single matrix-vector product with $A$, $A^T$, $L$ and $L^T$. In fact, if we include a re-orthogonalization step in \cref{alg:GKS} it is readily checked that the dominant cost is precisely the same as for \cref{alg:PNTM_gen}, i.e. given by \cref{eq:flops2}. We leave out the details here, but rather refer to \cite{LAMPE20122845} for more information. To get good performance with the GKS method it is also important to implement an efficient scalar root-finder for line \ref{rootfinder}. 

\begin{algorithm}
\caption{The GKS algorithm \cite{LAMPE20122845}}
\label{alg:GKS} 
\begin{algorithmic}[1] \small
\STATE{Construct initial $V_l\in\mathbb{R}^{n\times l}$ with $\mathcal{R}(V_l) = \mathcal{K}_l(A^TA,A^Tb)$ and $V_l^TV_l=I_l$}
\FOR{$k = l,l + 1,l + 2,\ldots,\text{until convergence}$} 
\STATE{Calculate $\alpha_k$ and the solution $y_k$ of \cref{eq:ygks} such that $||AV_ky_k - b||=\sigma$. \label{rootfinder}}
\STATE{Calculate new basis vector $v_{new}$ by \cref{eq:vnew} (and re-orthogonalize).}
\STATE{Expand basis $V_{k+1} = [V_k,v_{new}]$.}
\ENDFOR
\STATE{$x_k = V_k y_k$}
\end{algorithmic}
\end{algorithm}

Let us briefly comment on the difference between the GKS method and the Projected Newton method (for general form Tikhonov regularization). While both methods use Generalized Krylov subspaces, the way  the iterates are calculated is significantly different.  While the iterates generated by the Projected Newton method only satisfy the discrepancy principle (with a certain accuracy) after a certain number of iterations, the GKS constructs iterates $x_k$ that all satisfy $||Ax_k - b|| = \sigma$.  Moreover, in each iteration of the GKS algorithm, the projected minimization problem is also solved exactly, while the Projected Newton method only performs a single Newton iteration for each dimension $k$. The basis $V_{k}$ constructed in \cref{alg:GKS} is thus similar to the basis generated in \cref{alg:PNTM_gen}, although not the same since both methods compute different iterates. 

Lastly, we briefly mention that the GKS method also closely resembles the Generalized Arnoldi Tikhonov (GAT) method \cite{GAZZOLA2014180}. The latter method uses the Arnoldi algorithm to construct a basis $V_k$ for the Krylov subspace $\mathcal{K}_k(A,b)$, which implies that it can only be applied to square matrices $A$. In each iteration of the GAT method, \cref{eq:ygks} is solved for a single value of $\alpha_k$ and then the regularization parameter is updated using a single step of the secant method based on the discrepancy principle. The Krylov subspace basis is subsequently expanded using one step of the Arnoldi algorithm. In contrast to the GKS method, the intermediate iterates $x_k$ in the GAT method do not necessarily need to satisfy discrepancy principle. 

\subsection{Generalized Krylov subspace method for $\ell_1$ regularization}

The GKSpq method developed in \cite{doi:10.1137/140967982} combines the iteratively reweighted norm (IRN) algorithm \cite{rodriguez2008efficient,4380459} with ideas from the Generalized Krylov subspace method described above. It is able to solve the regularized problem with $\ell_q$ data fidelity term and $\ell_p$ regularization term: 
\begin{equation}\label{eq:lplq}
\min_{x\in\mathbb{R}^n} \frac{1}{q} ||A x - b||_q^q + \frac{\alpha}{p} ||L x||_p^p.
\end{equation}
The authors in \cite{doi:10.1137/140967982} present the algorithm for a fixed regularization parameter $\alpha$. However, similarly as for \cref{alg:GKS}, it is possible to formulate a version that finds a suitable regularization parameter that satisfies the discrepancy principle, i.e. this approach can be used to solve \cref{eq:noise_lplq}. For the sake of presentation, however, we focus our discussion on the simplified case with fixed regularization parameter $\alpha$ and with $q=2,p=1$.

The main idea of the IRN approach is to replace the non-differentiable term $||Lx||_1$ with a sequence of $\ell_2$ norm approximations $||W_k L x||_2^2$ for $k\geq0$ with a (diagonal) weighting matrix $W_k \in\mathbb{R}^{s\times s}$. To do so, we first consider the matrix
\begin{equation} \label{eq:lirn}
\bar{W}(Lx) = \text{diag} \left( \frac{1}{\sqrt{ | [L x]_i |}} \right)_{i=1,\ldots,n}  
\end{equation}
where $ [L x]_i $ is the notation for the $i$th component of the vector $Lx\in\mathbb{R}^s$. Obviously \cref{eq:lirn} is only well-defined if $[Lx]_i\neq0$ for all $i$. In that case we have
\begin{equation*}
||\bar{W}(Lx)Lx||^2_2 = \sum_{i = 1}^{n} \left(\frac{[Lx]_i}{\sqrt{ | [Lx]_i |}}\right)^2 = \sum_{i = 1}^{n}\frac{[Lx]^2_i}{| [Lx]_i |} = \sum_{i = 1}^{n} |[Lx]_i| = ||Lx||_1.
\end{equation*}
To avoid division by zero, we slightly alter the definition of $\bar{W}(Lx)$, as follows: 
\begin{equation}\label{eq:ltilde}
W(Lx) = \text{diag} \left( \frac{1}{\sqrt{\tau([Lx]_i)}} \right)_{i=1,\ldots,n}, \hspace{0.2cm} \text{with}\hspace{0.2cm} \tau(t) = \begin{cases} |t| & \text{if }|t| \geq \tilde{\tau} \\ \tilde{\tau} & \text{if } |t|< \tilde{\tau} \end{cases}. 
\end{equation}
where $\tilde{\tau}$ is a small positive constant.

Let $V_1 = A^Tb / ||A^Tb||$ be the initial basis for the GKSpq method and choose initial point $x_0$. In iteration $k\geq 1$ the GKSpq method computes the solution of the projected problem
\begin{equation}\label{eq:ygkspq}
y_k  = \argmin_{y\in\mathbb{R}^{k}} ||AV_{k}y - b||_2^2 + \alpha ||W_k LV_{k}y||_2^2,
\end{equation}
with the diagonal weighting matrix $W_k = W(Lx_{k-1})$ defined by \cref{eq:ltilde}. The next iterate is then given by $x_k = V_ky_k$.
This problem can be solved efficiently by considering QR decompositions $AV_k = Q_k R_k$ and $W_k L V_k = \bar{Q}_k \bar{R}_k$. The solution of \cref{eq:ygkspq} satisfies the $k\times k$ linear system
\begin{equation} \label{eq:solvegkspq}
\left(R_k^T R_k + \alpha \bar{R}_k^T \bar{R}_k \right)y_k = R_k^T Q_k^T b. 
\end{equation}

Next, the basis is extended by adding the normalized residual of the unreduced problem. The residual is given by
\begin{eqnarray}\label{eq:vnew2}
\tilde{v}_{new} &=& \left(A^T A + \alpha L^T W_k^2 L\right)V_{k} y_k - A^Tb \\
&=& A^T(AV_ky_k - b) + \alpha L^T W_k^2 (LV_k y_k)
\end{eqnarray}
A full description of the GKSpq method is given by \cref{alg:GKSpq}. 
\subsubsection{Computational cost of the GKSpq method.}
Let us briefly comment on the computational cost of a single iteration of the algorithm. Each iteration requires a single matrix-vector product with $A$, $A^T$, $L$ and $L^T$, namely on lines \ref{matvec1} and \ref{matvec2}. Computing the matrix $W_kLV_k$ requires $sk$ \textsc{flop}s.  Note that the QR decomposition of $AV_k$ can be efficiently updated throughout the algorithm using  \cref{eq:updateqr1}-\cref{eq:rtr} and only requires about $4mk$ \textsc{flop}s. However, since the matrix $W_k$ changes in each iterations, we need to recompute the QR decomposition for $W_kLV_k \in \mathbb{R}^{s\times k}$ from scratch each iteration. This requires approximately $2sk^2$ \textsc{flop}s. This is the main difference with the number of \textsc{flop}s required for an iteration of the Projected Newton method, see \cref{eq:flops}. Calculating $AV_ky_k$ and $LV_ky_k$ takes about $2mk$ and $2sk$ \textsc{flop}s respectively. To summarize, we have that the cost of a single iteration of \cref{alg:GKSpq} is given by
\begin{equation}
\textsc{flop}s(A) + \textsc{flop}s(A^T) + \textsc{flop}s(L) + \textsc{flop}s(L^T) +  6mk + 3sk + 2 sk^2,  
\end{equation}
where we again disregard terms that do not depend on the iteration index $k$ or that do not have a factor $m$, $n$ or $s$. 
When we compare this with \cref{eq:flops} we can observe that, once the term $sk^2$ starts to dominate, the Projected Newton method requires much fewer \textsc{flop}s per iteration than the GKSpq method. 
 
To conclude this section, we like to mention another interesting hybrid method that combines the IRN approach with a projection step on a lower dimensional subspace. In \cite{doi:10.1137/130917673,doi:10.1137/18M1194456} the authors combine the IRN approach with flexible Krylov subspace methods for the $\ell_p$ regularized problem \cref{eq:lp}, but it only works for invertible matrices $L$. 

\begin{algorithm}
\caption{The GKSpq method for the solution of the $\ell_1$ regularized problem}
\label{alg:GKSpq} 
\begin{algorithmic}[1] \small
\STATE{Initialize $v_0 = A^T b/||A^T b||; V_1 = v_0; AV_1 = Av_0; LV_1=Lv_0$; $W_1 = W(Lx_0);$}.
\FOR{$k = 1,2,\ldots,\text{until convergence}$} 
\STATE{Compute $R_k$ and $Q_k$ for the QR decomposition of $AV_k$ using \cref{eq:updateqr1}-\cref{eq:rtr}}
\STATE{Compute $\bar{R}_k$ and $\bar{Q}_k$ for the QR decomposition of $W_k LV_k$}
\STATE{Compute $y_k$ the solution of \cref{eq:solvegkspq}.}
\STATE{Compute weighting matrix $W_{k+1} = W(LV_{k}y_{k})$ defined by \cref{eq:ltilde}.}
\STATE{Compute residual $\tilde{v}_{new} = A^T(AV_ky_k - b) + \alpha L^TW_k^2(LV_k y_{k})$. \label{matvec1}}
\STATE{Compute new basis vector $v_{new} = \tilde{v}_{new}/||\tilde{v}_{new}||$}
\STATE{$V_{k+1} = [V_k,v_{new}]; AV_{k+1} = [AV_k, Av_{new}]; LV_{k+1} = [LV_k,Lv_{new}];$ \label{matvec2}}
\ENDFOR
\STATE{$x_k = V_ky_k;$}
\end{algorithmic}
\end{algorithm}

\section{Numerical experiments}\label{sec:num}

In this section we perform a number of experiments to illustrate the behavior of the Projected Newton method and to compare it to the reference methods described in \cref{sec:refmeth}. We start with some small scale toy models to illustrate different interesting properties and then consider larger, more representative test-problems to study the quality of the obtained solution. The main focus of this section is on the sparsity inducing $\ell_1$ norm and the edge preserving total variation regularization term, rather than the general form Tikhonov problem, which we only briefly discuss. We also comment on possible different convergence criteria that can be used.  All experiments are performed using MATLAB. 

\subsection{Generalized Krylov subspaces and induced tridiagonal structure}

\paragraph{\textbf{Experiment 1.}}
We start with a small experiment that justifies calling the basis $V_k$ generated in the Projected Newton method a Generalized Krylov subspace, see \cref{eq:vtilde} and \cref{eq:gs} and the discussion following these expressions. This experiment is purely for theoretical interest and illustrating the behavior of the Generalized Krylov subspaces used in the algorithm and should not be regarded as a representative test-problem. 

We consider the test-problem \texttt{heat} from the MATLAB package Regularization Tools \cite{hansen1994regularization} with $n=200$ and optional parameter $kappa = 3$ which gives a matrix $A\in\mathbb{R}^{200\times 200}$ with condition number $\kappa(A)\approx 5.8 \times 10^{17}$ and exact right-hand side $b_{ex}$. Next, we add $15\%$ Gaussian noise to the exact right-hand side $b_{ex}$, which means we have $\tilde{\epsilon}:=||b - b_{ex}||/||b_{ex}|| = 0.15$. We refer to this value as the relative noise-level. Moreover, we choose parameters $\lambda_0 = 10^5$, $\eta=1$ and stop the algorithm when $||F(x_k,\lambda_k)|| < 10^{-6}$. We apply the Projected Newton method to the standard form Tikhonov problem, to the general form Tikhonov problem with $L \in\mathbb{R}^{(n-1)\times n}$ the forward finite difference operator given by \cref{eq:forward} and the regularized problem with $\Psi(x) = \Psi_{1}(x)$ as defined in \cref{eq:smoothlp}, the smooth approximation to the $\ell_1$ norm with $\beta = 10^{-5}$. 

In \cref{fig:tridiag} (top) we plot the convergence history in terms of $||F(x_k,\lambda_k)||$ for the three regularized problems and we show that the regularization parameter $\lambda_k$ stabilizes quickly. Here and in what follows the index $k$ denotes the iteration index. As explained in \cref{sec:projminsec}, we know that the basis $V_k$ generated by the Projected Newton method for the standard form Tikhonov problem is in fact a Krylov subspace basis due to the shift-invariance property of Krylov subspaces \cref{eq:shiftinv}. As a consequence we have that $V_{k}^T (A^T A + \alpha_k I) V_k$ with $\alpha_k = 1/\lambda_k$ is a tridiagonal matrix. In \cref{fig:tridiag} (bottom) we have illustrated this by showing the absolute value of the elements of this matrix for the final iteration $k$. Similarly, for the general form Tikhonov problem we show the matrix $V_{k}^T (A^T A + \alpha_k L^T L) V_k$ and for the $\ell_1$ regularized problem we show $V_{k}^T (A^T A + \alpha_k \nabla^2 \Psi_1(x_k)) V_k$ (each with their respective basis $V_k$ and parameter $\alpha_k$). For the latter two problems, this matrix is not really tridiagonal since $V_k$ is not an actual Krylov subspace, but due to the rapid stabilization of the regularization parameter, we can observe that the size of the elements on the three main diagonals is much larger than the size of the other elements (see the color-bar right of the corresponding figure). This shows that the matrix $V_k$ does closely resemble an actual Krylov subspace basis in this particular case.

Again, we do not claim that this is a representative for all problems, but it does illustrates our observations. The effect may of course be less pronounced for other, more realistic test-problems.

\begin{figure}
\centering
\begin{subfigure}{1\linewidth}
\begin{tabular}{cc}
\includegraphics[width=0.49\textwidth]{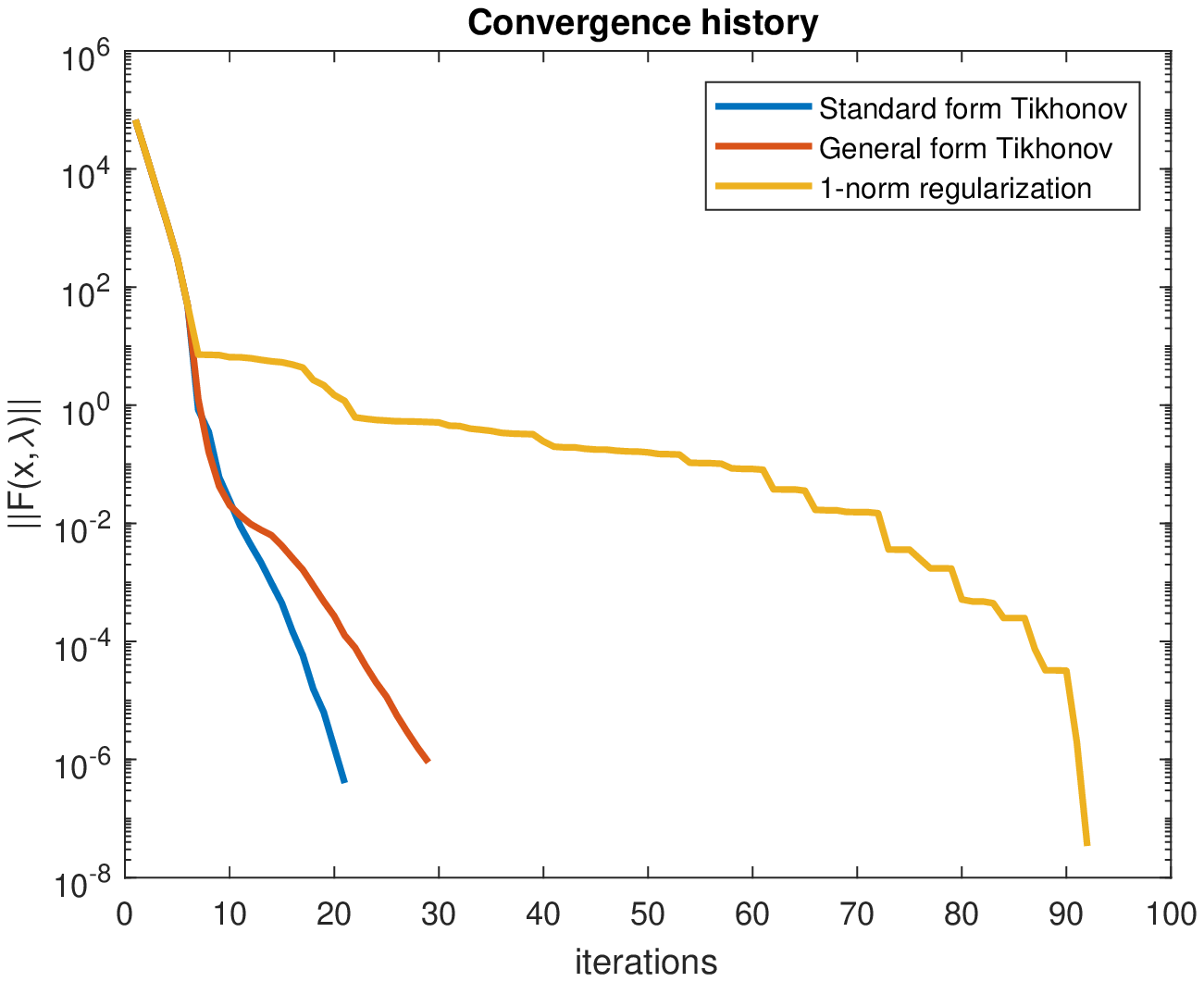} & \includegraphics[width=0.49\textwidth]{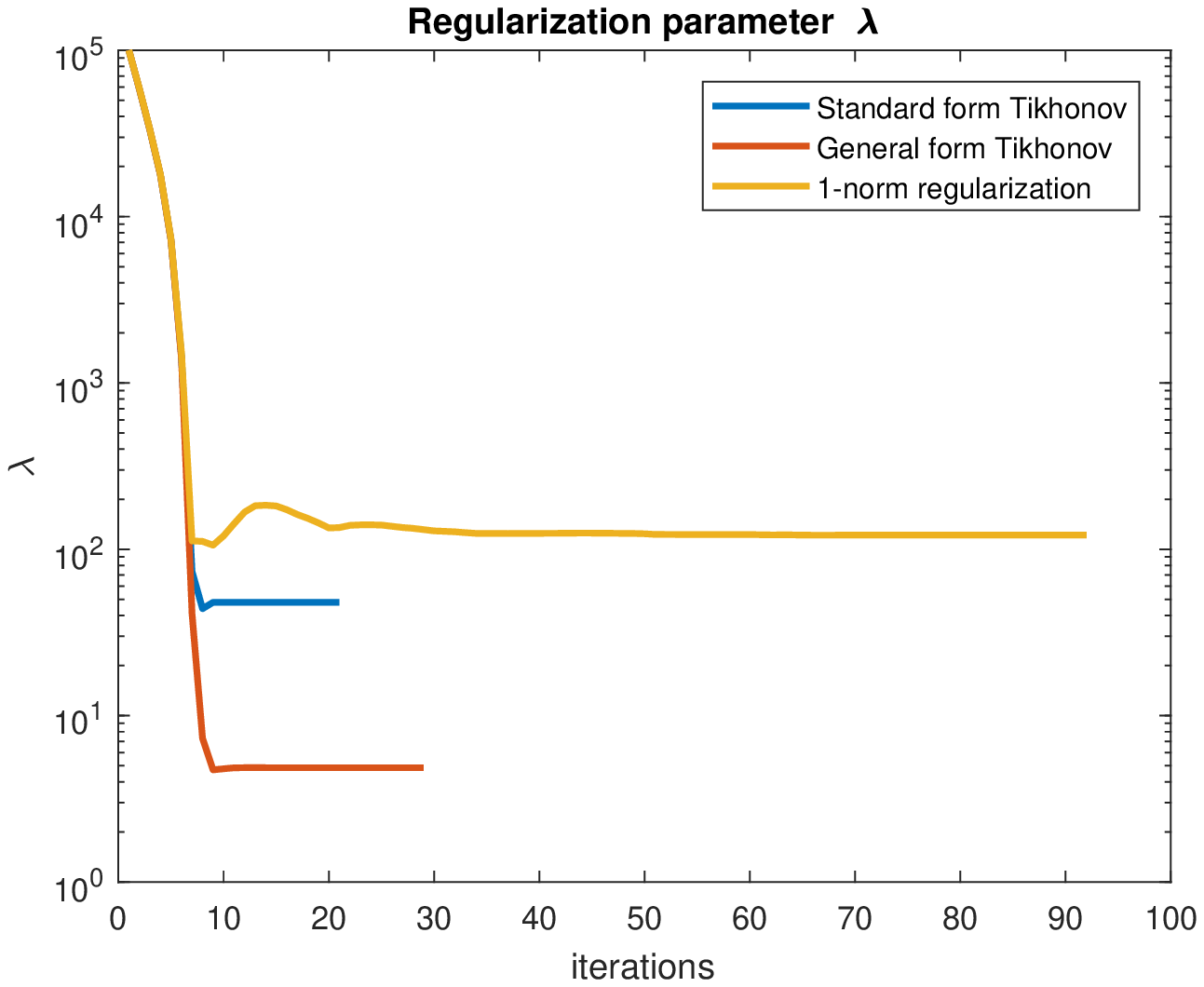} 
\end{tabular}
\setcounter{subfigure}{0}%
\end{subfigure}

\begin{subfigure}{1\linewidth}
\begin{tabular}{ccc}
\hspace{-0.5cm}\includegraphics[width=0.36\textwidth]{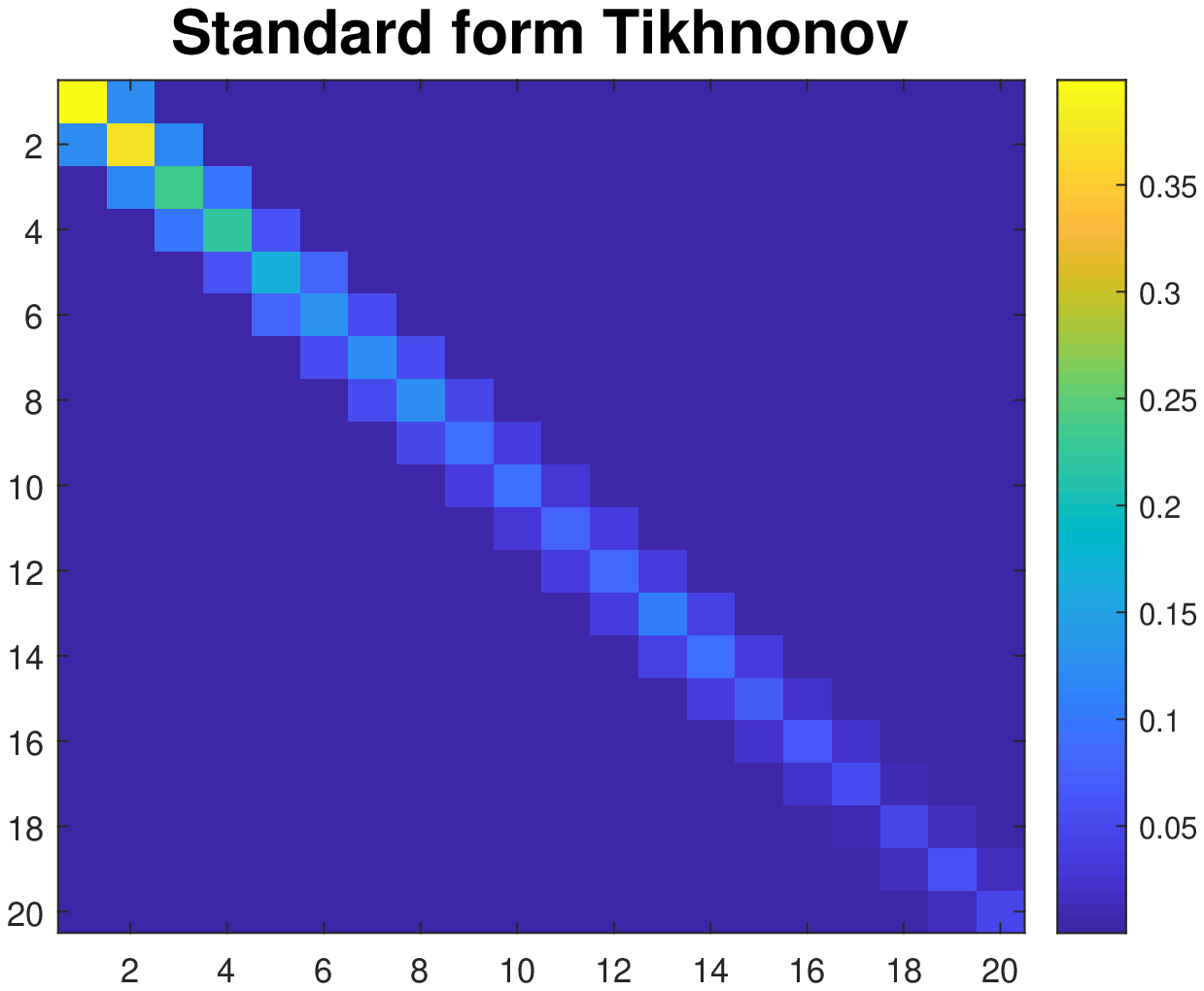} \hspace{-0.9cm} & \includegraphics[width=0.36\textwidth]{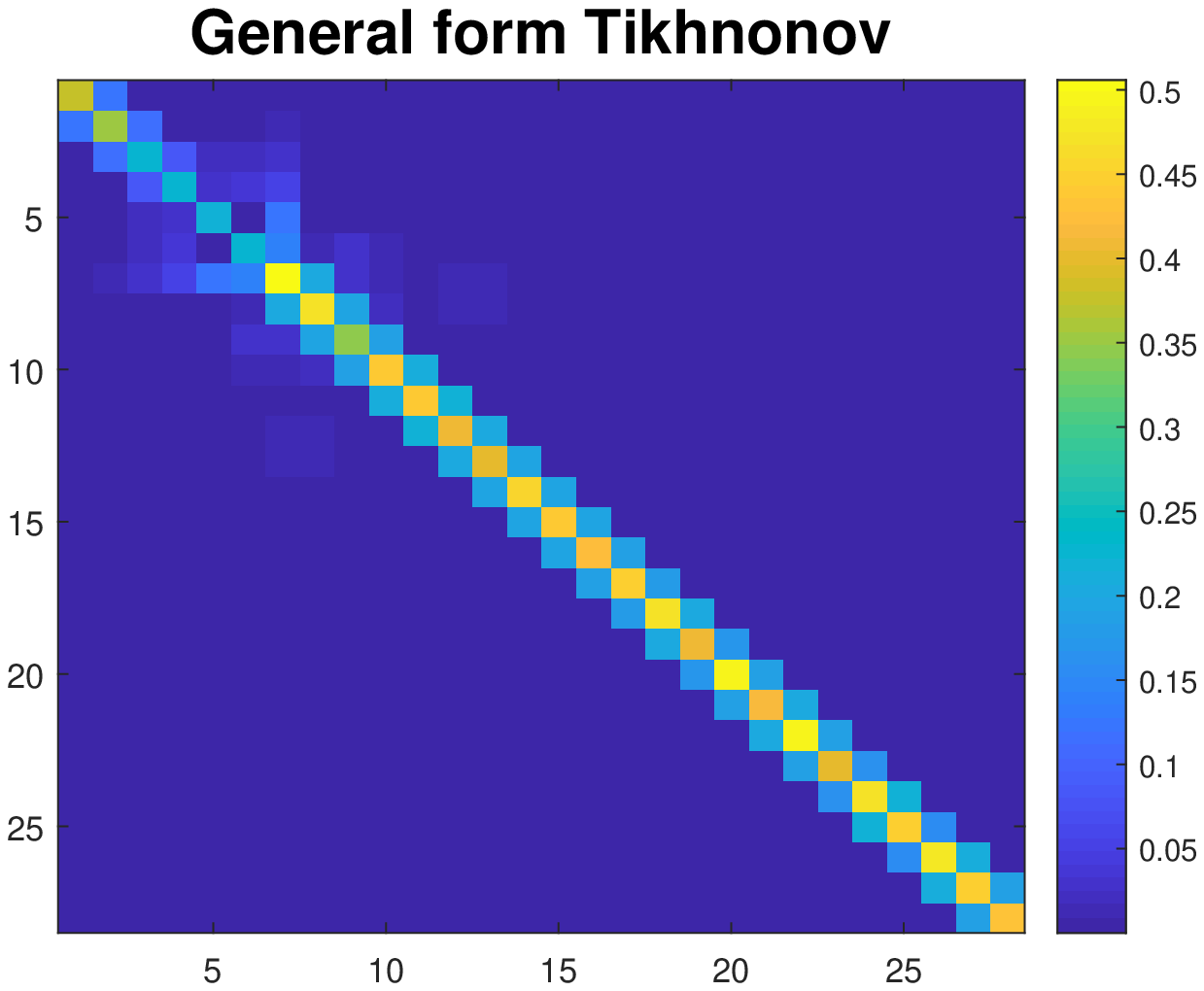} \hspace{-0.9cm} & \includegraphics[width=0.36\textwidth]{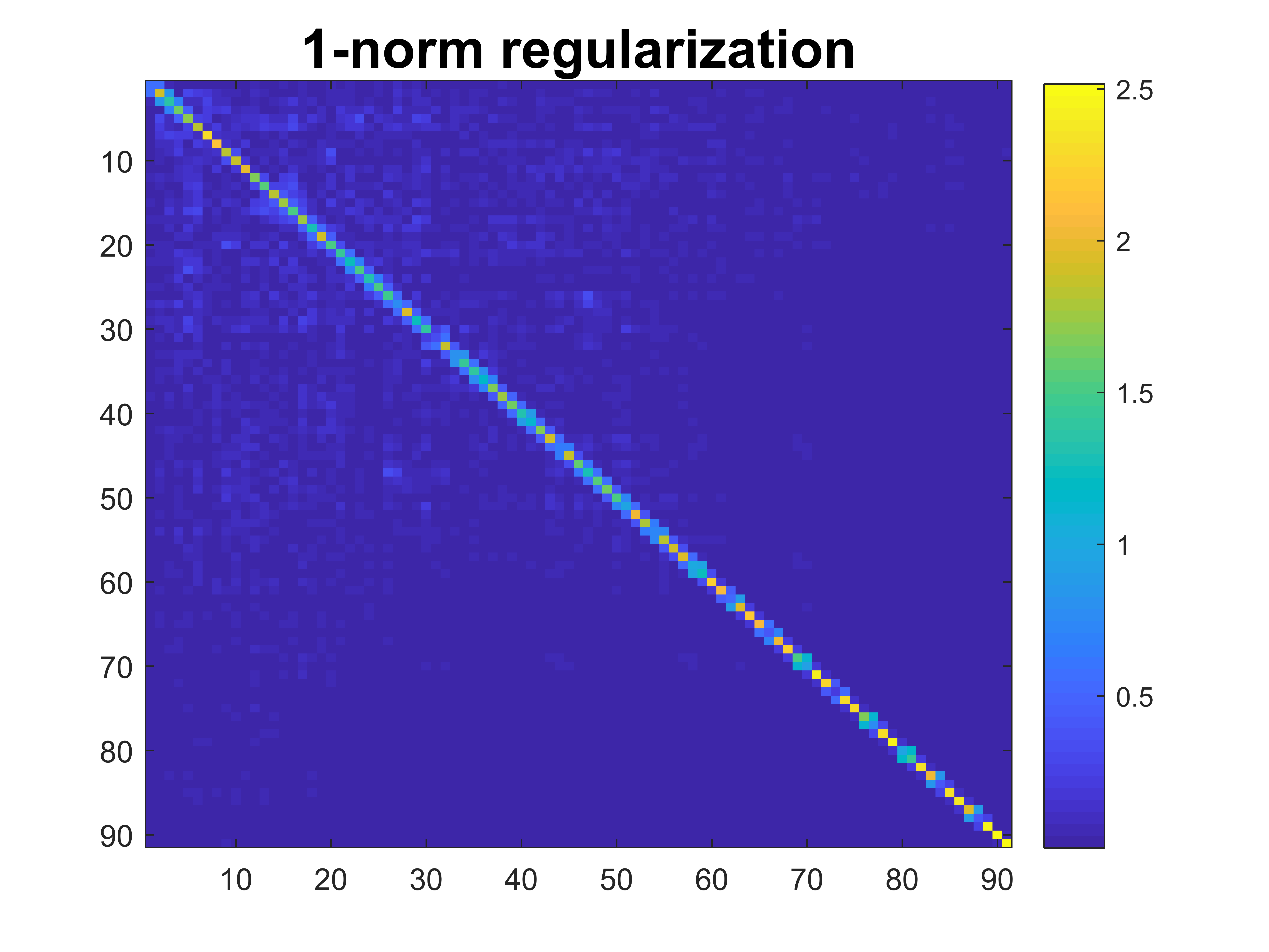}
\end{tabular}
\setcounter{subfigure}{2}%
\end{subfigure}
\caption{\textit{\textbf{Experiment 1.}} (Top) Convergence history (left) and regularization parameter (right) for standard form Tikhonov, general form Tikhonov and $\ell_1$ norm regularization (i.e. $\Psi(x) = \Psi_1(x)$) applied to test-problem \texttt{heat} with $A\in\mathbb{R}^{200 \times 200}$. $15\%$ Gaussian noise is added to the exact right-hand side $b_{ex}$. (Bottom) Approximate tridiagonal structure of the projected regularized normal equations induced by the Generalized Krylov subspace. \label{fig:tridiag}}
\end{figure}

\subsection{Comparison with GKS}

\begin{figure}
\begin{tabular}{cc}
\includegraphics[width=0.49\textwidth]{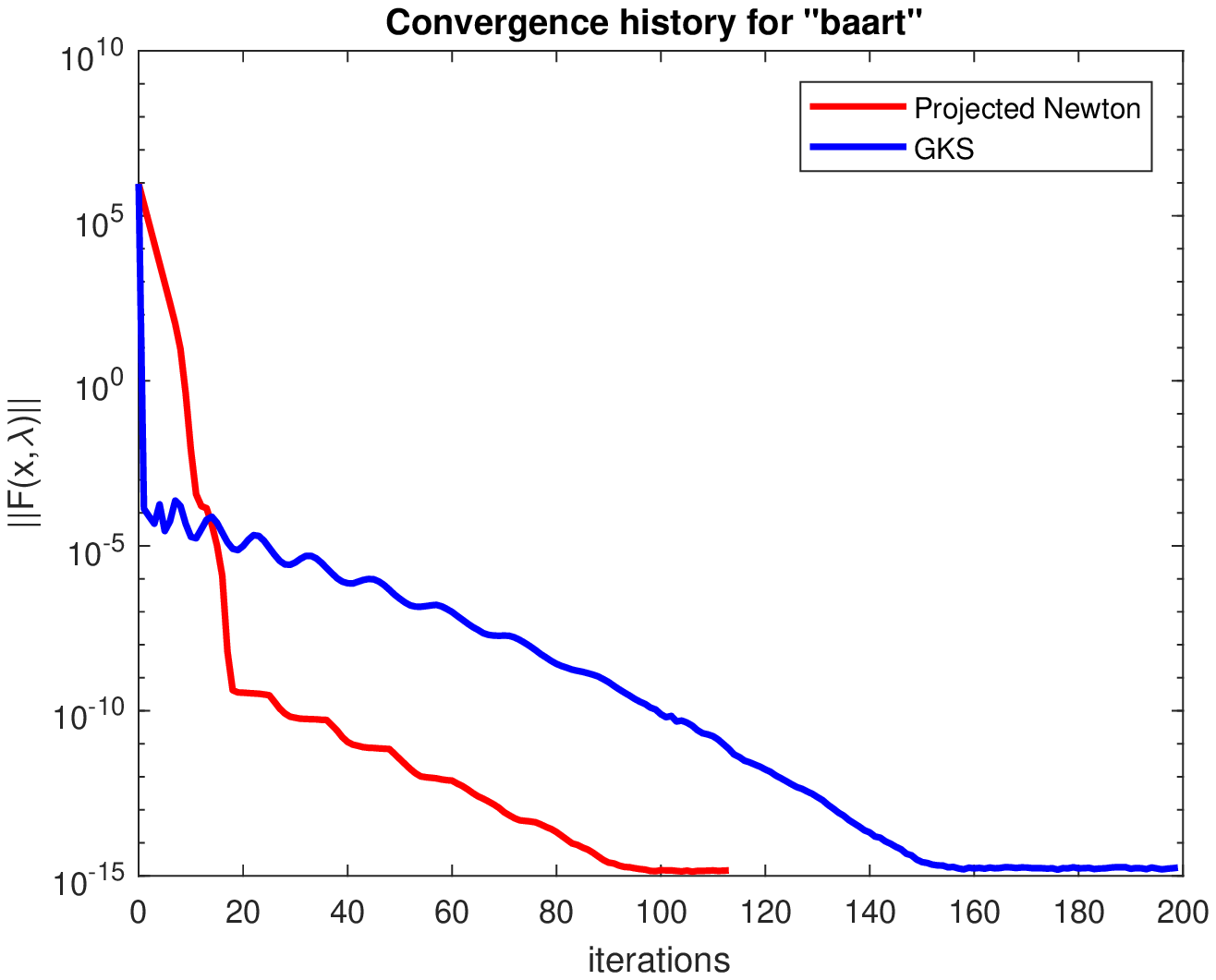} & \includegraphics[width=0.49\textwidth]{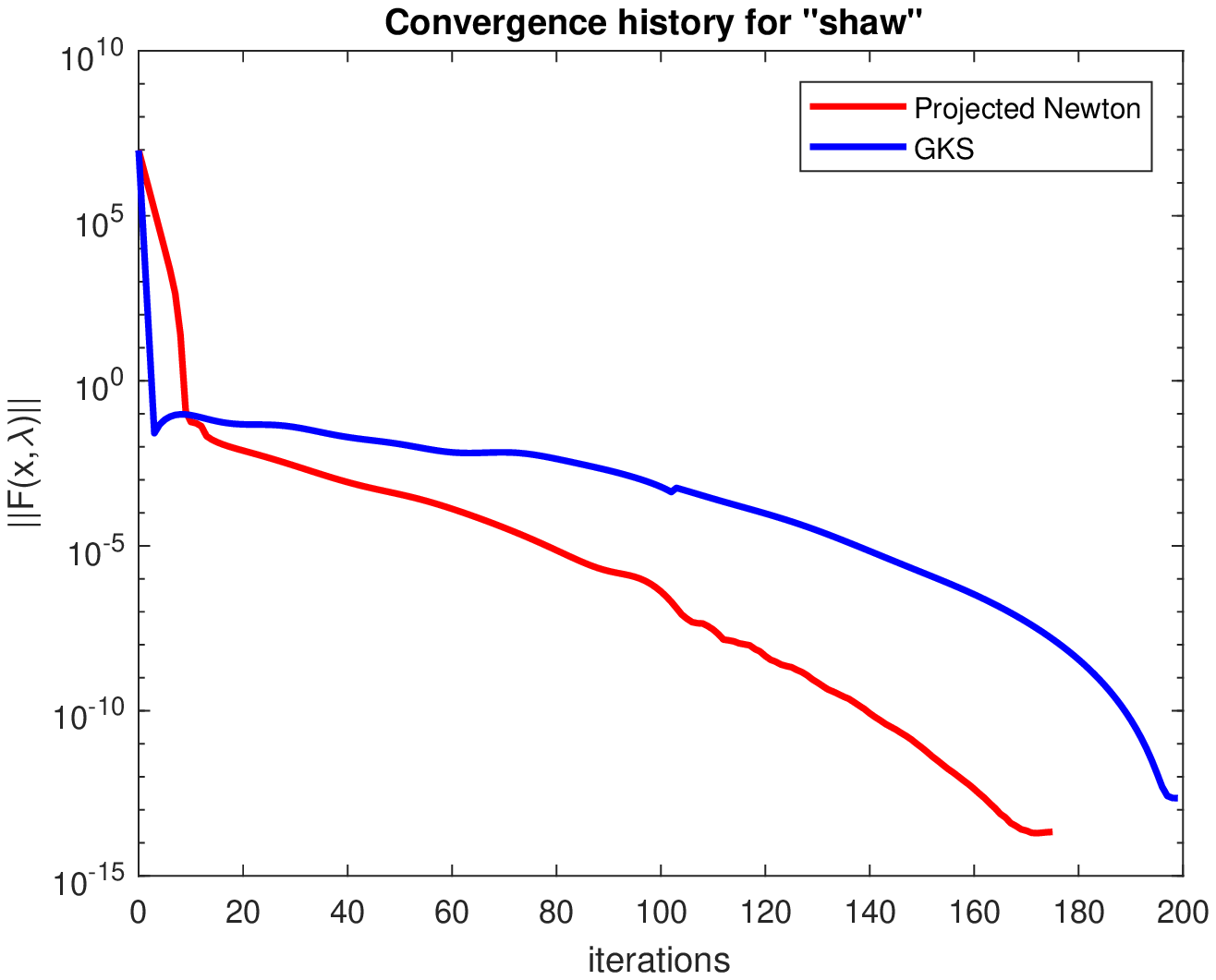} 
\end{tabular}
\caption{\textit{\textbf{Experiment 2.}} Comparison of convergence history of Projected Newton method and the Generalized Krylov subspace method (\cref{alg:GKS}) in terms of number of iterations for general form Tikhonov problem with $L \in\mathbb{R}^{(n-1)\times n}$ the finite difference operator given by \cref{eq:forward}. \label{fig:gks}}
\end{figure}

\paragraph{\textbf{Experiment 2.}} In the next experiment we compare the Projected Newton method with the GKS method, see \cref{alg:GKS}. To do so, we consider the general form Tikhonov problem and again take $L \in\mathbb{R}^{(n-1)\times n}$ the finite difference operator \cref{eq:forward} and test-problems \texttt{baart} and \texttt{shaw} from the Regularization Tools package with $n=200$, which gives us two matrices $A\in\mathbb{R}^{200 \times 200}$ that both have a condition number $\kappa(A) \approx 10^{19}$. We consider a relative noise-level of $0.1$ and keep the other parameters for the Projected Newton method the same as before. The dominant cost per iteration for the GKS method and Projected Newton method is comparable, as discussed in \cref{sec:gkssec}, so we show the convergence of $||F(x_k,\lambda_k)||$ in terms of the number of iterations\footnote{In our implementation of \cref{alg:GKS} used in this experiment we use a regula falsi scalar root-finder on line \ref{rootfinder}, which is not very efficient. Hence, we deem it more appropriate to compare iteration count rather than run-time.}.  The result of this experiment is given by \cref{fig:gks}. The Projected Newton method requires fewer iteration to converge to machine precision accuracy than the GKS method for both these test-problems. The effect is most pronounced for the test-problem \texttt{baart}, where the Projected Newton method reaches machine precision accuracy for $||F(x_k,\lambda_k)||$ in about $100$ iterations, while it takes the GKS method about $150$ iterations to reach the same accuracy. It is also interesting to note that the convergence for both methods is very rapid in the first few iterations and then slows down considerably.

\subsection{Influence of the smoothing parameter $\beta$}

\begin{figure}
\centering
\begin{subfigure}{.49\linewidth}
\includegraphics[width=\linewidth]{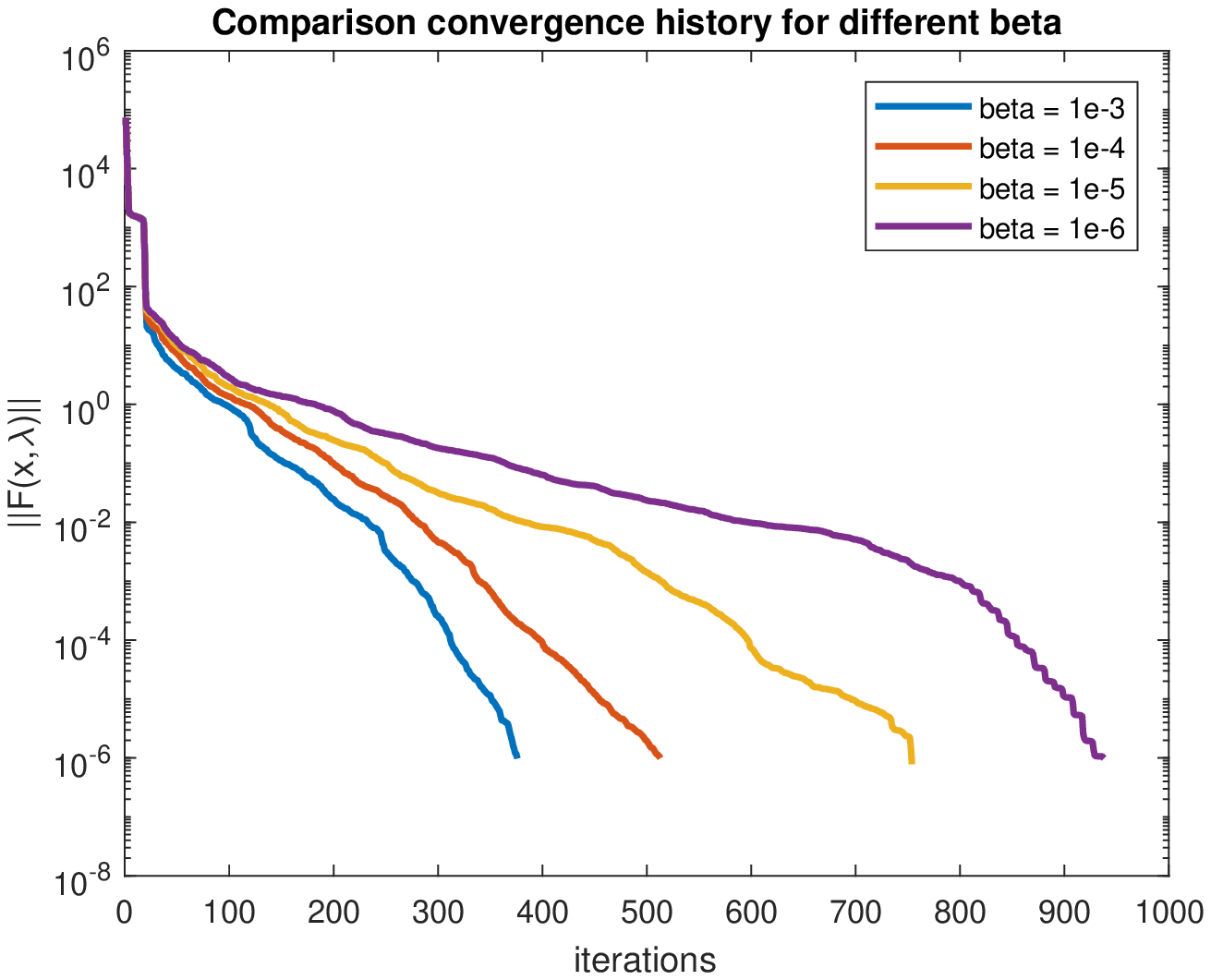}
\setcounter{subfigure}{0}%
\end{subfigure}
\begin{subfigure}{.49\linewidth}
\includegraphics[width=\linewidth]{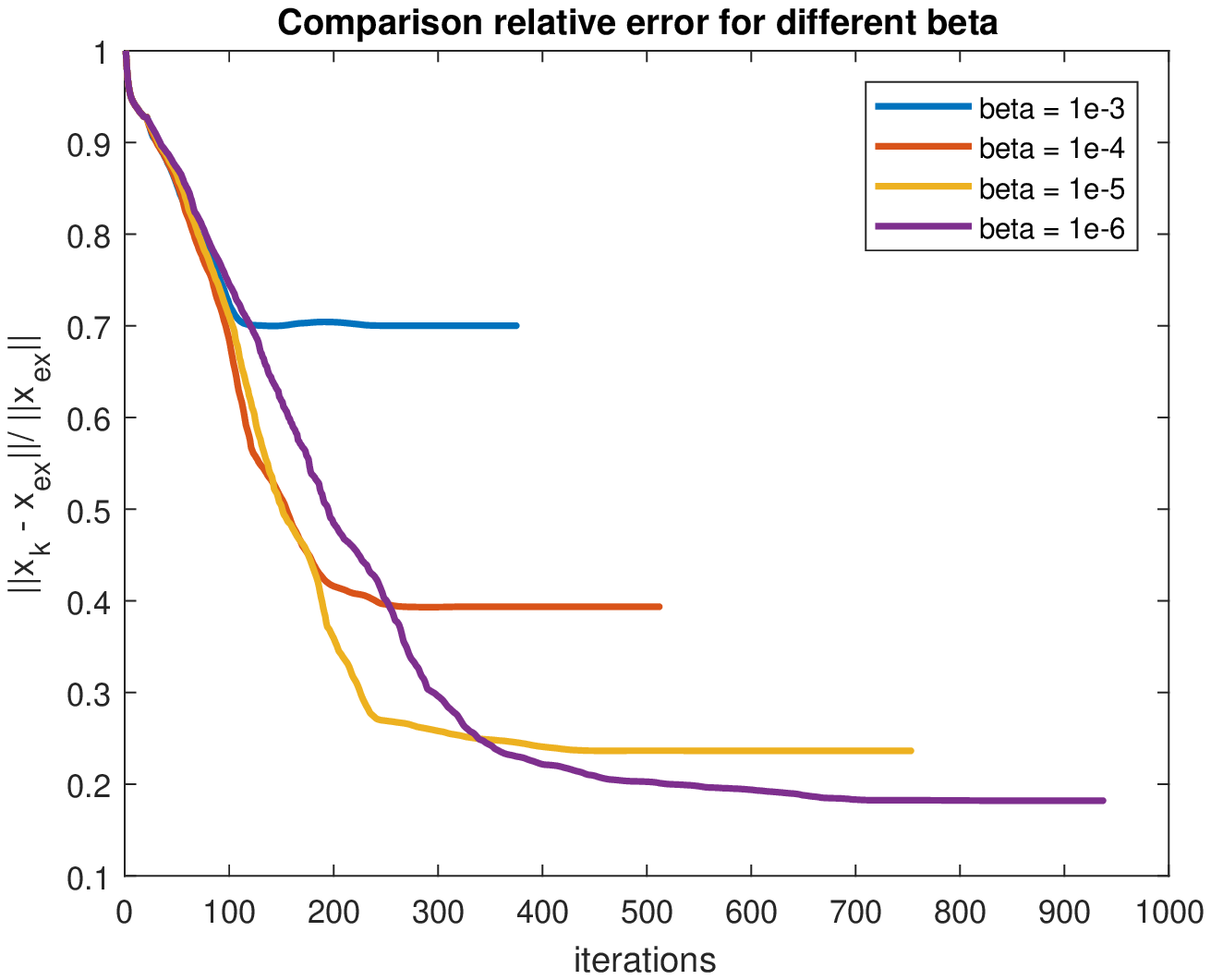}
\end{subfigure}

\begin{subfigure}{1\linewidth}
\includegraphics[width=\linewidth]{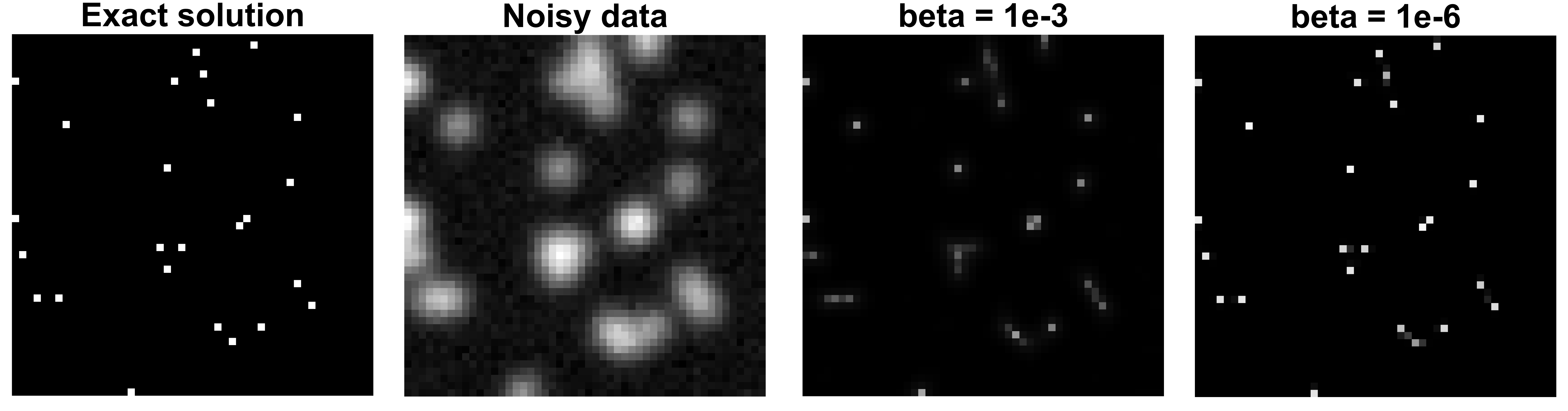}
\setcounter{subfigure}{2}%
\end{subfigure}

\caption{\textit{\textbf{Experiment 3.}} (Top) Comparison of convergence history (left) and relative error (right) for different values of the smoothing parameter $\beta$ for the image deblurring problem \texttt{PRblurgauss} with $\Psi(x) = \Psi_1(x)$. (Bottom) Exact solution and noisy data for deblurring experiment. Reconstructed solutions for $\beta = 10^{-3}$ and $\beta=10^{-6}$ are also shown. \label{fig:beta}}
\end{figure}

In this section we look at the quality of the obtained solution and how well the smooth function $\Psi_p(x)$ actually approximates $\frac{1}{p}||x||_p^p$. We study this for test-problems with a sparse exact solution, such that the (approximate) $\ell_1$ regularized solution, i.e. with $\Psi(x) = \Psi_1(x)$, should give a good reconstruction.

\paragraph{\textbf{Experiment 3.}} Let us consider a small sparse image $X \in \mathbb{R}^{50 \times 50}$ with approximately $1\%$ of the pixels set to one. We take the deblurring matrix $A \in \mathbb{R}^{2500 \times 2500}$ provided by the function \texttt{PRblurgauss} from the IR Tools MATLAB package \cite{gazzola2019ir} (with optional parameter \textit{blurlevel} set to \textit{mild}). This function generates the forward model $A$ that simulates an image deblurring problem with Gaussian point spread function. The exact solution $x_{ex}$ is obtained by stacking all columns $X$ and we get the corresponding exact data $b_{ex} = Ax_{ex}$. We add $10\%$ Gaussian noise to the data. See \cref{fig:beta} for an illustration of the exact solution and noisy data. To study the quality of the reconstruction we consider different values of $\beta$ used in the definition of the smooth approximation $\Psi_1(x)$, more precisely we take $\beta = 10^{-3},10^{-4},10^{-5}$ and $10^{-6}$. The smaller this value, the closer $\Psi_1(x)$ is to the actual $\ell_1$ norm. We stop the algorithm when $||F(x_k,\lambda_k)||< 10^{-6}$ and keep the other parameters the same as before. 

The result of this experiment is given by \cref{fig:beta}. In the top left figure we show the value $||F(x_k,\lambda_k)||$ for the different choices of $\beta$. We can observe that convergence is slower when $\beta$ is smaller. This is not entirely surprising since a small value of $\beta$ also implies that $\Psi_1(x)$ becomes less smooth and that the condition number of the Hessian $\nabla^2 \Psi_1(x)$ can become larger. However, a smaller value of $\beta$ also gives a better reconstruction, as shown by the relative error in the right figure. In \cref{fig:beta} we also show the actual reconstruction for $\beta = 10^{-3}$ and $\beta = 10^{-6}$. It is clear that the latter image is indeed a better reconstruction since in that case $\Psi_1(x)$ is a better approximation of $||x||_1$. When choosing this value, one should try to find a balance between improved quality of the reconstruction and the amount of work needed to find the solution. In our experience we find that $\beta=10^{-5}$ often performs quite well in practice. For very ill-conditioned matrices, however, it might be more suitable to consider slightly larger values, for instance $\beta = 10^{-4}$. Lastly, we would also like the point out the stable behavior of the error on the top right plot of \cref{fig:beta} and the fact that the error has stabilized well before a very accurate solution has been found (in terms of $||F(x_k,\lambda_k)||$). We explore this last observation in a bit more detail in \cref{sec:stopping}. 

In \cite{wu2019signal} the authors perform an interesting similar experiment where they study the influence of the smoothing parameter ($\mu$ in their notation) for five different smooth approximations of the $\ell_1$ norm, including our choice of $\Psi_1(x)$. The algorithm they study is based on the nonlinear Conjugate Gradient method \cite{8161135}. They make similar observations regarding the smoothing parameter: ``the smaller the parameter $\mu$, the better the signal to recover; and the more cpu time and iterations have to spend accordingly." We refer to their paper for more information \cite{wu2019signal}.

\subsection{Comparison with GKSpq}

To asses the performance of the Projected Newton method we perform a few timing experiments comparing the run-time of the algorithm with that of the GKSpq method, see \cref{alg:GKSpq}. We consider four different test-problems from the IR Tools package. The first problem we consider is the image dublurring test-problem \texttt{PRblurshake}, where the matrix $A$ simulates random camera motion in the form of shaking. For the second problem we take another image deblurring test-problem, namely \texttt{PRblurrotation}. Here, the blurring effect simulated by the matrix $A$ is due to rotation of the object in the image. The other two test-problems we consider are tomography problems. We take a classical 2D Computed Tomography (CT) problem, namely \texttt{PRtomo}, and a spherical means tomography problem \texttt{PRspherical}, which arise for instance in photo-acoustic imaging. For more information on these test-problems we refer to \cite{gazzola2019ir}.

\begin{figure}
\centering
\begin{tabular}{cc}
\includegraphics[width=0.49\textwidth]{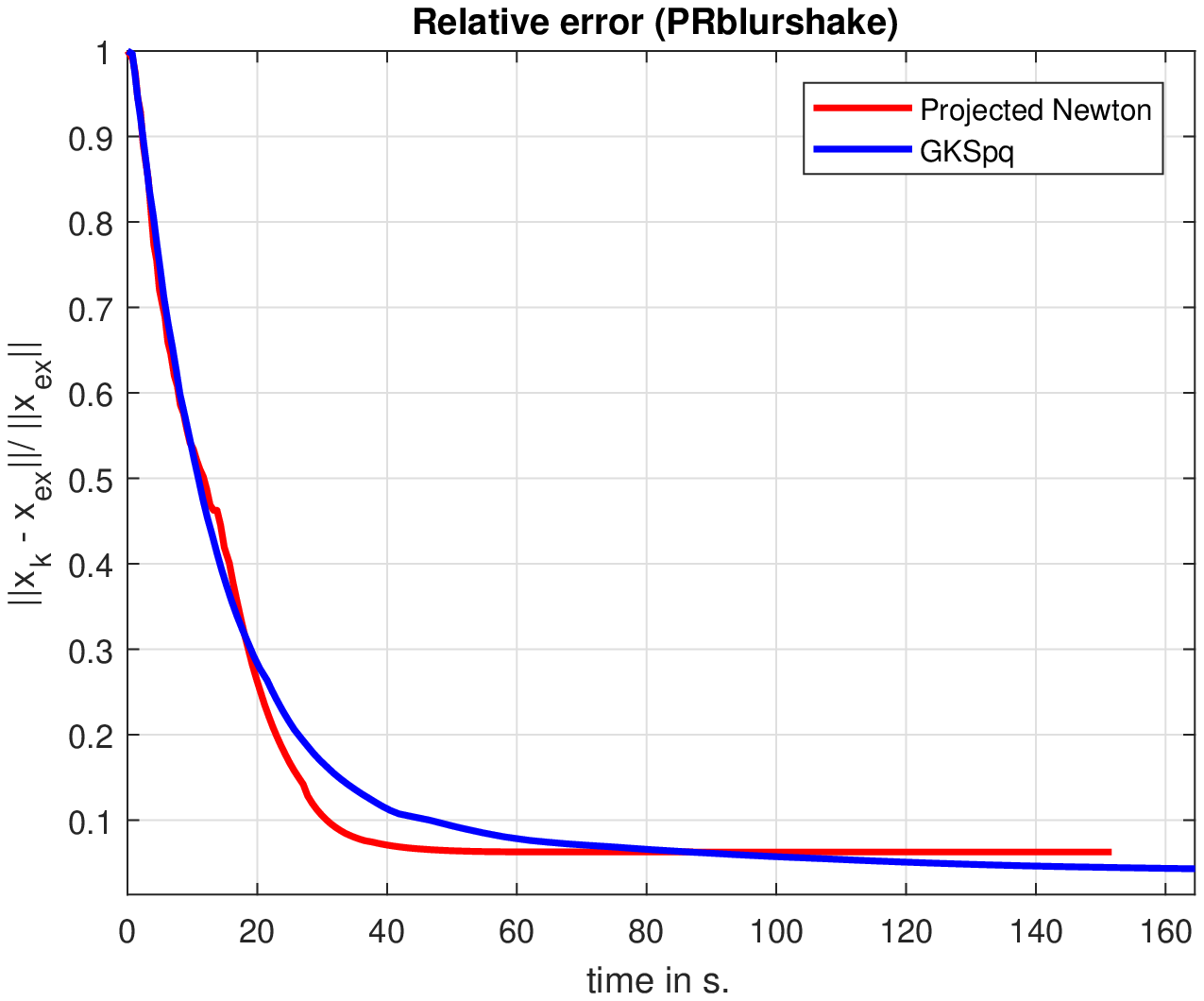} & \includegraphics[width=0.49\textwidth]{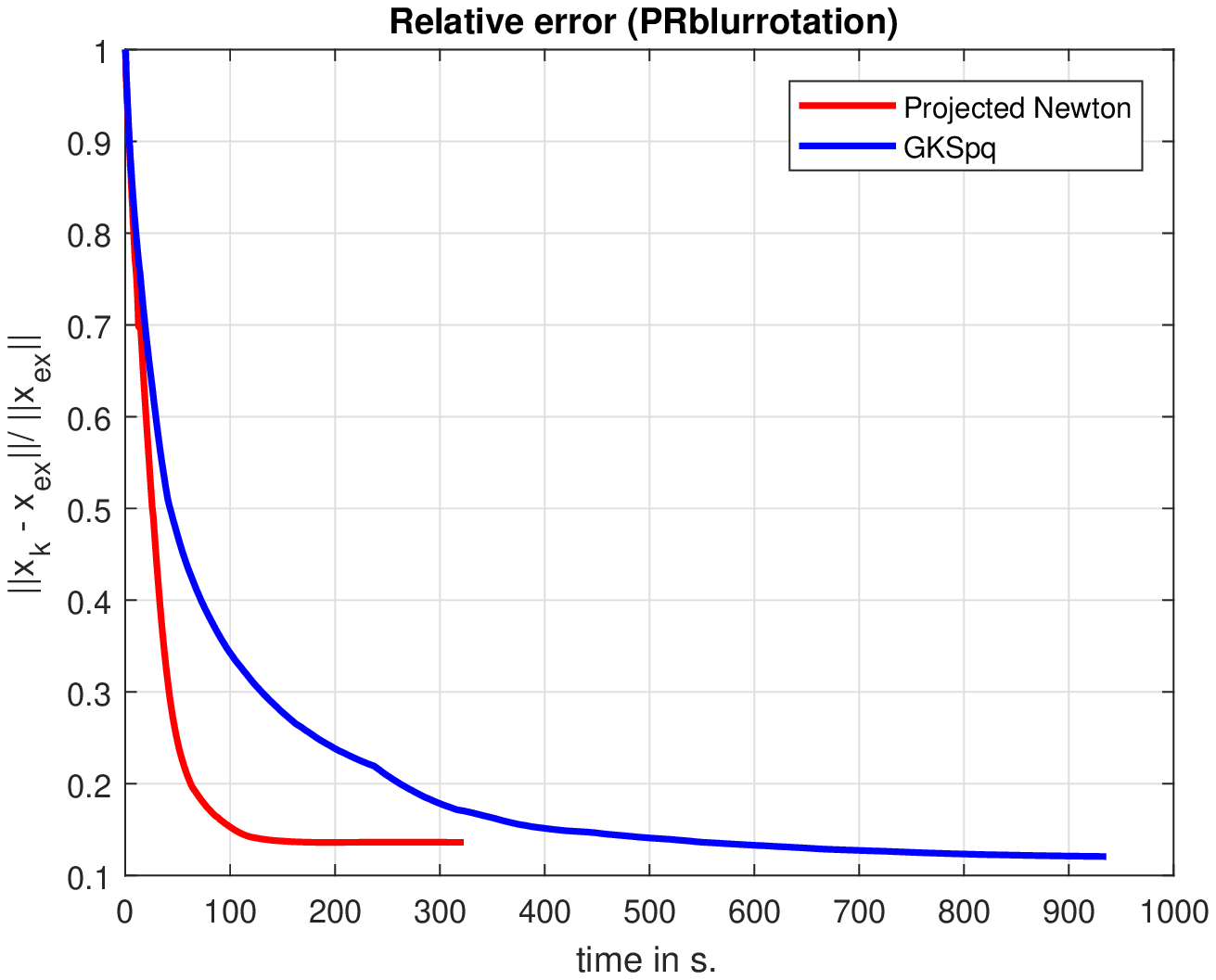} \\
\includegraphics[width=0.49\textwidth]{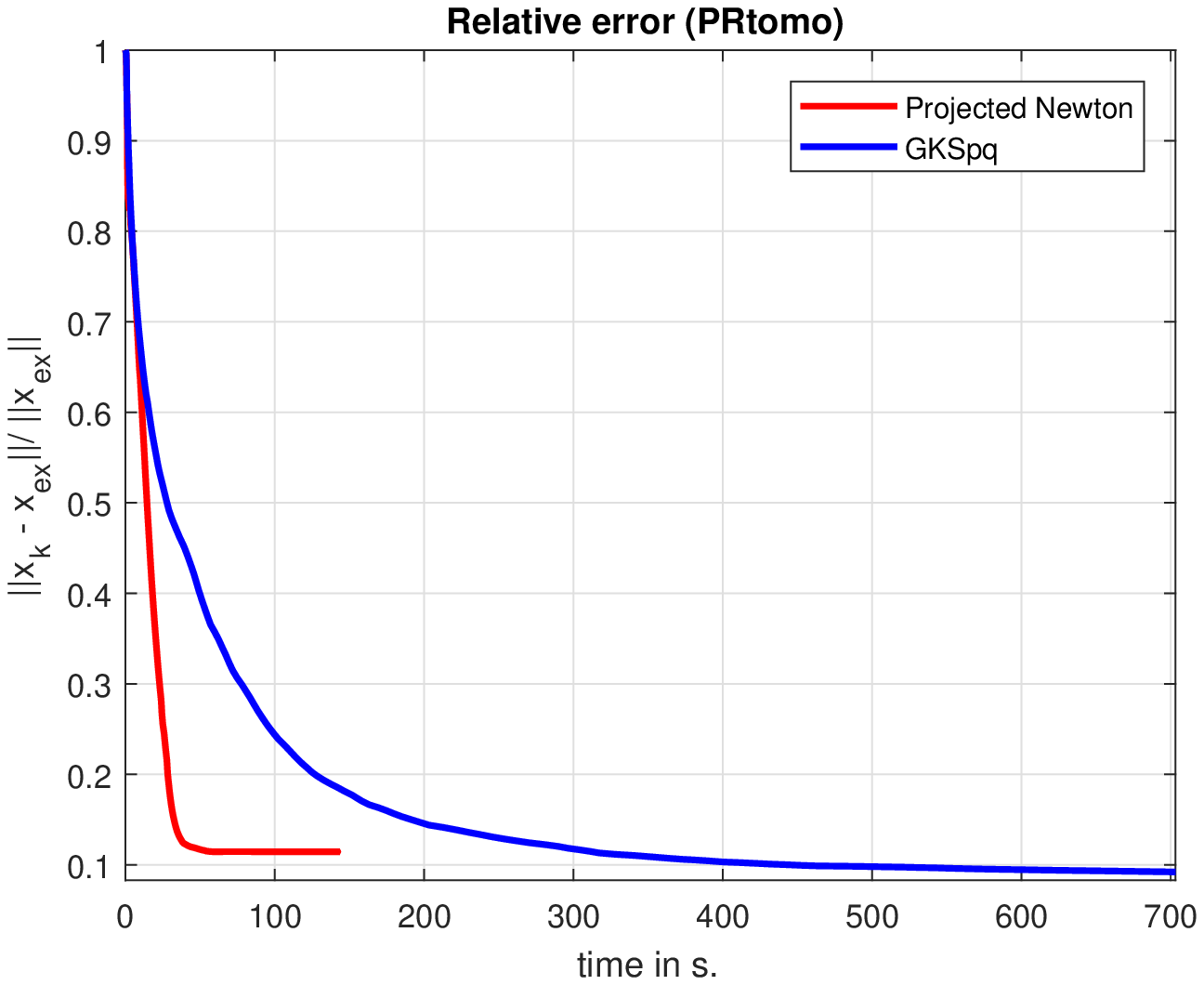} & \includegraphics[width=0.49\textwidth]{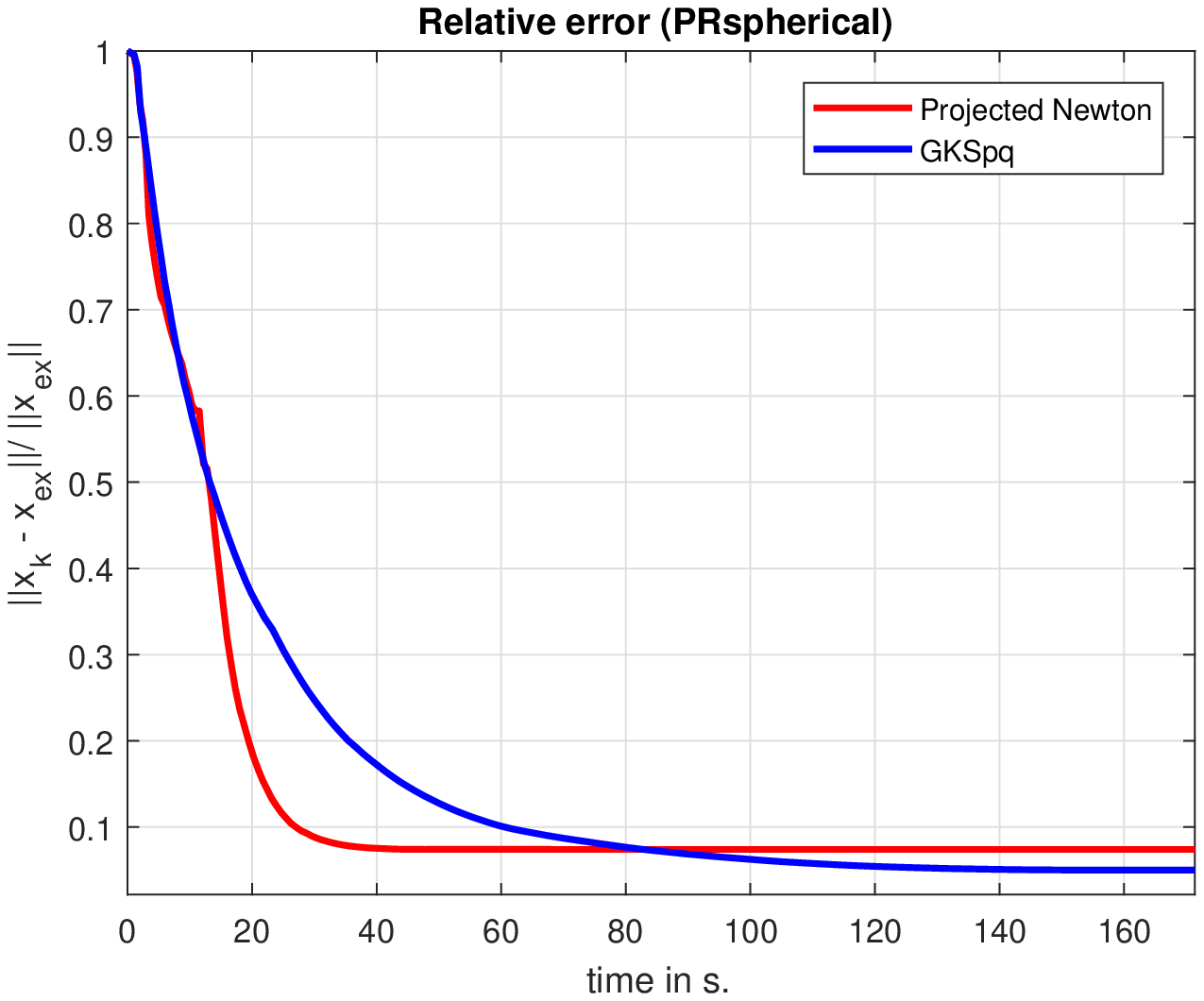} 
\end{tabular}
\caption{\textit{\textbf{Experiment 4.}} Comparison of relative error of the iterates of the Projected Newton method and the GKSpq method (\cref{alg:GKSpq}) in terms of run-time (in seconds) for sparse reconstruction problem, i.e. with $\Psi(x) = \Psi_1(x)$. \label{fig:timingsparse}}
\end{figure}

\paragraph{\textbf{Experiment 4.}} We obtain an exact solution $x_{ex}$ by stacking the columns of a sparse randomly generated image $X \in \mathbb{R}^{512 \times 512}$ with approximately $1\%$ of the pixels set to one, similarly to what we did for experiment 3. For a particular matrix $A$ (generated by one of the four functions described in the previous paragraph) we construct the exact right-hand side $b_{ex}=Ax_{ex}$ and add $10\%$ Gaussian noise. For the image deblurring test-problems \texttt{PRblurshake} and \texttt{PRblurrotation} the matrix has size $262,144 \times 262,144$, while the tomography problems \texttt{PRtomo} and \texttt{PRspherical} have matrices with dimension $ 130,320 \times 262,144$ and $370,688 \times 262,144$ respectively. We solve these test-problems with a sparsity inducing regularization term, i.e. we choose $\Psi(x) = \Psi_1(x)$ for the Projected Newton method and $L=I_n$ for the GKSpq method. For the latter method we take the regularization parameter $\alpha = 1/\lambda_k$, where $\lambda_k$ is the final Lagrange multiplier given by the Projected Newton method.

We let both algorithms run for a fixed number of iterations and compare the relative error of the iterates generated by both the Projected Newton method and GKSpq method as a function of the run-time (in seconds). For the Projected Newton method we take $\lambda_0=10^5$ and $\beta=10^{-5}$. For the GKSpq algorithm we take $\tilde{\tau} = 10^{-4}$ in the definition of the weighting matrix, see \cref{eq:ltilde}, which seems to give the best result. The result of this experiment is given by \cref{fig:timingsparse}. The relative error for the Projected Newton method for these test-problems always stagnates before the relative error of the GKSpq method. The effect is relatively small for \texttt{PRblurshake}, but the Projected Newton method converges much more rapidly than the GKSpq method for the other three test-problems. However, the GKSpq method seems the give slightly more accurate results in terms of the final obtainable relative error. 

\begin{figure}
\begin{center}
\includegraphics[width=1\textwidth]{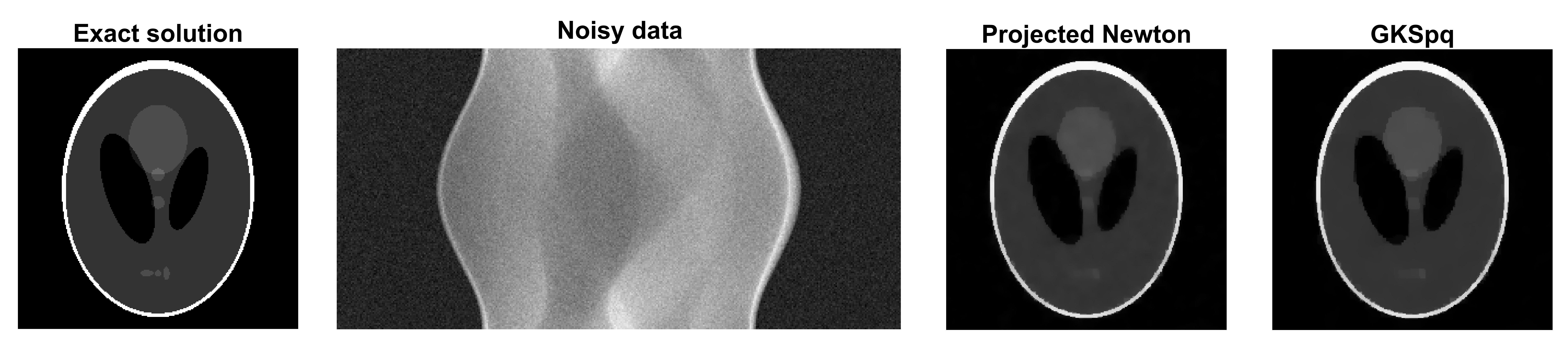} 
\end{center}
\caption{\textit{\textbf{Experiment 5.}} The exact solution, noisy data and reconstructed solution obtained by the Projected Newton method and GKSpq method for the test-problem $\texttt{PRtomo}$ with total variation regularization. \label{fig:solutions_tomo}}
\end{figure}

\paragraph{\textbf{Experiment 5.}} For the next performance experiment we consider a similar set-up as in experiment $4$, but we now use total variation regularization, i.e. we choose $\Psi(x) = \Psi_1(Lx)$ the smooth approximation of $\text{TV}(x)$, see \cref{eq:tv}, with $L$ defined as \cref{eq:ltv}. We take the well-known shepp-logan phantom with $256\times 256$ pixels as exact image and again obtain an exact solution $x_{ex}$ by stacking the columns of the image. We add $10\%$ Gaussian noise to the exact data $b_{ex}=Ax_{ex}$, where the matrix $A$ is again generated using the same four IR tool functions. The regularization matrix $L$ has size $130,560 \times 65,536$, while the matrices $A$ now have dimensions $65,536\times 65,536$ for the deblurring test-problems and dimensions $65,160\times 65,536$ and  $92,672\times 65,536$ for the problems \texttt{PRtomo} and \texttt{PRsperical} respectively. See \cref{fig:solutions_tomo} for an illustration of the exact image $x_{ex}$ and noisy data $b$ for \texttt{PRtomo}. The rest of the experiment is the same as experiment 4. We only change $\tilde{\tau} = 10^{-3}$ since this gives a better result for GKSpq. The result of this experiment is given by \cref{fig:timingtv}. We again see that the relative error for the Projected Newton method stagnates well before the GKSpq method and that both algorithms produce solutions of similar quality. As an example we have shown the reconstructed solution for \texttt{PRtomo} for both methods in \cref{fig:solutions_tomo}. The Projected Newton method and GKSpq produce images of similar quality, but the former method converges much more rapidly. 

\begin{figure}
\centering
\begin{tabular}{cc}
\includegraphics[width=0.49\textwidth]{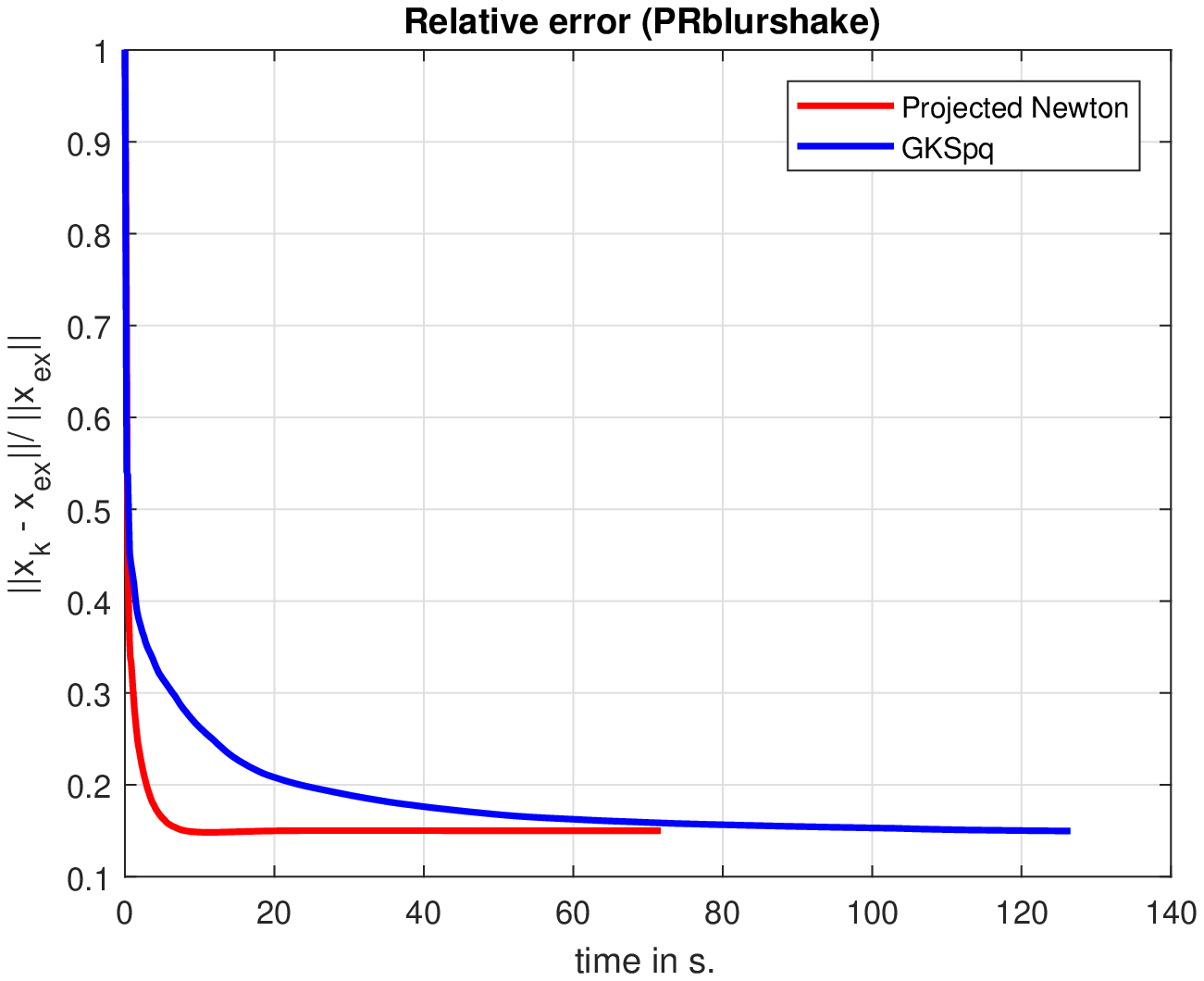} & \includegraphics[width=0.49\textwidth]{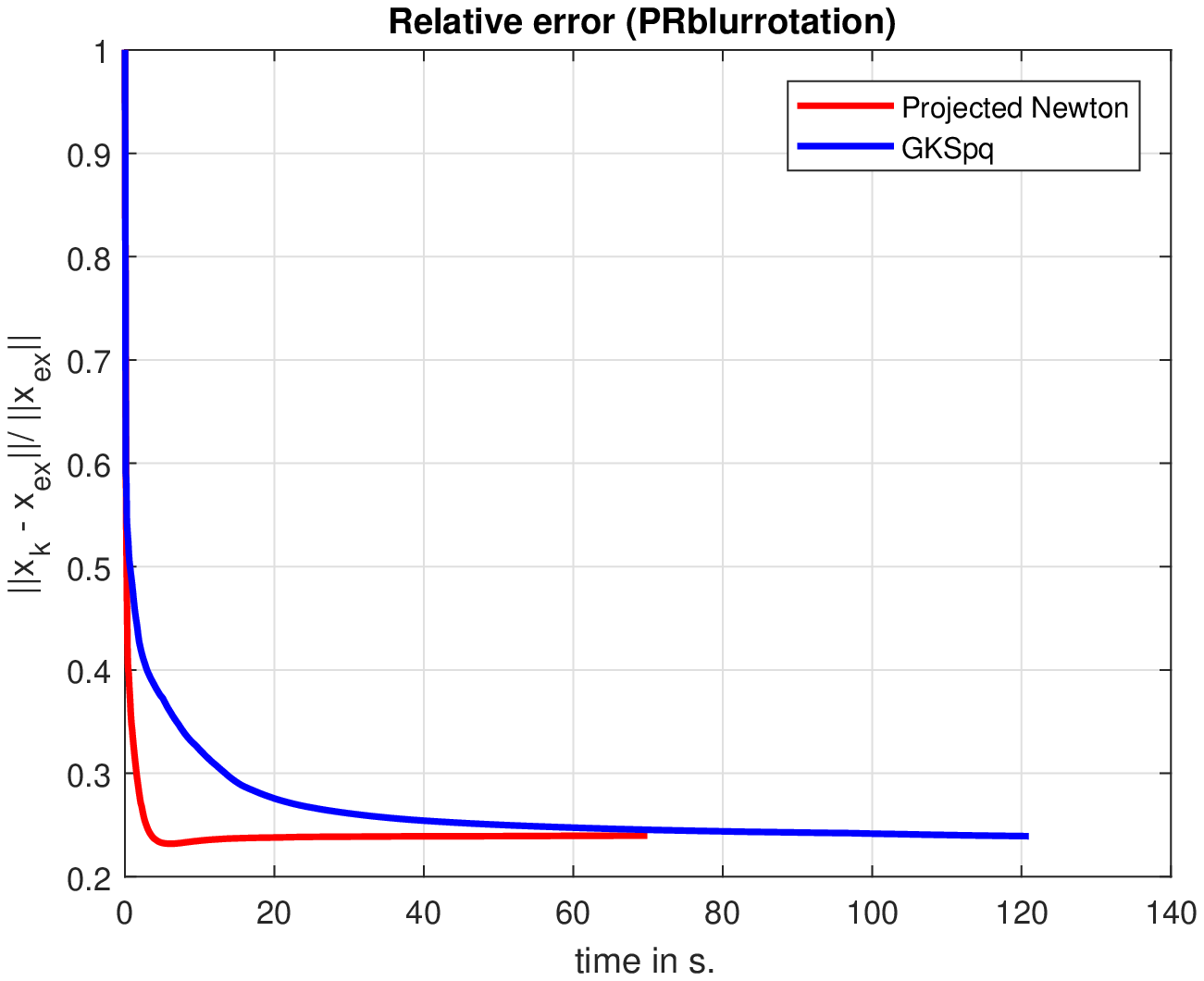} \\
\includegraphics[width=0.49\textwidth]{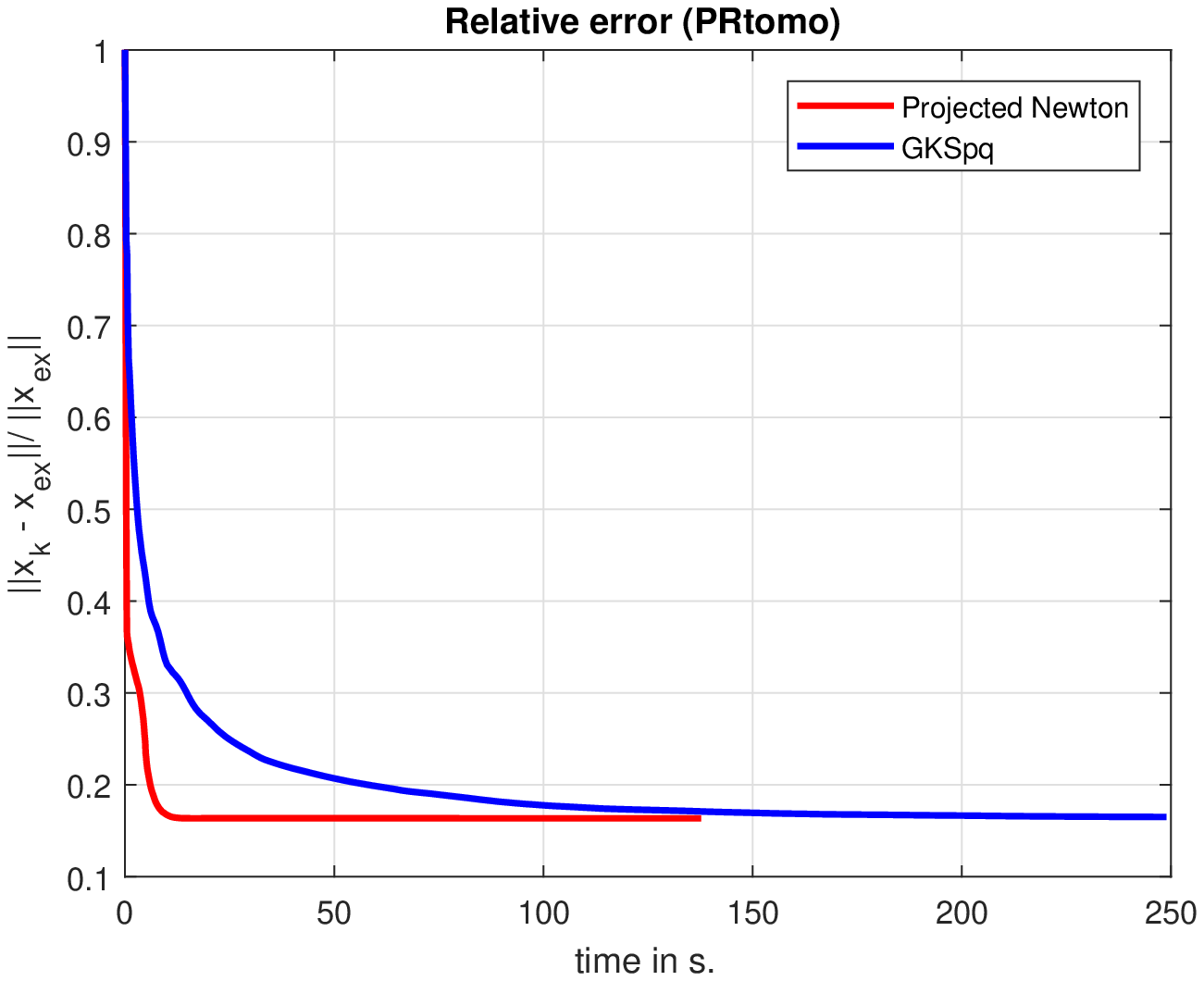} & \includegraphics[width=0.49\textwidth]{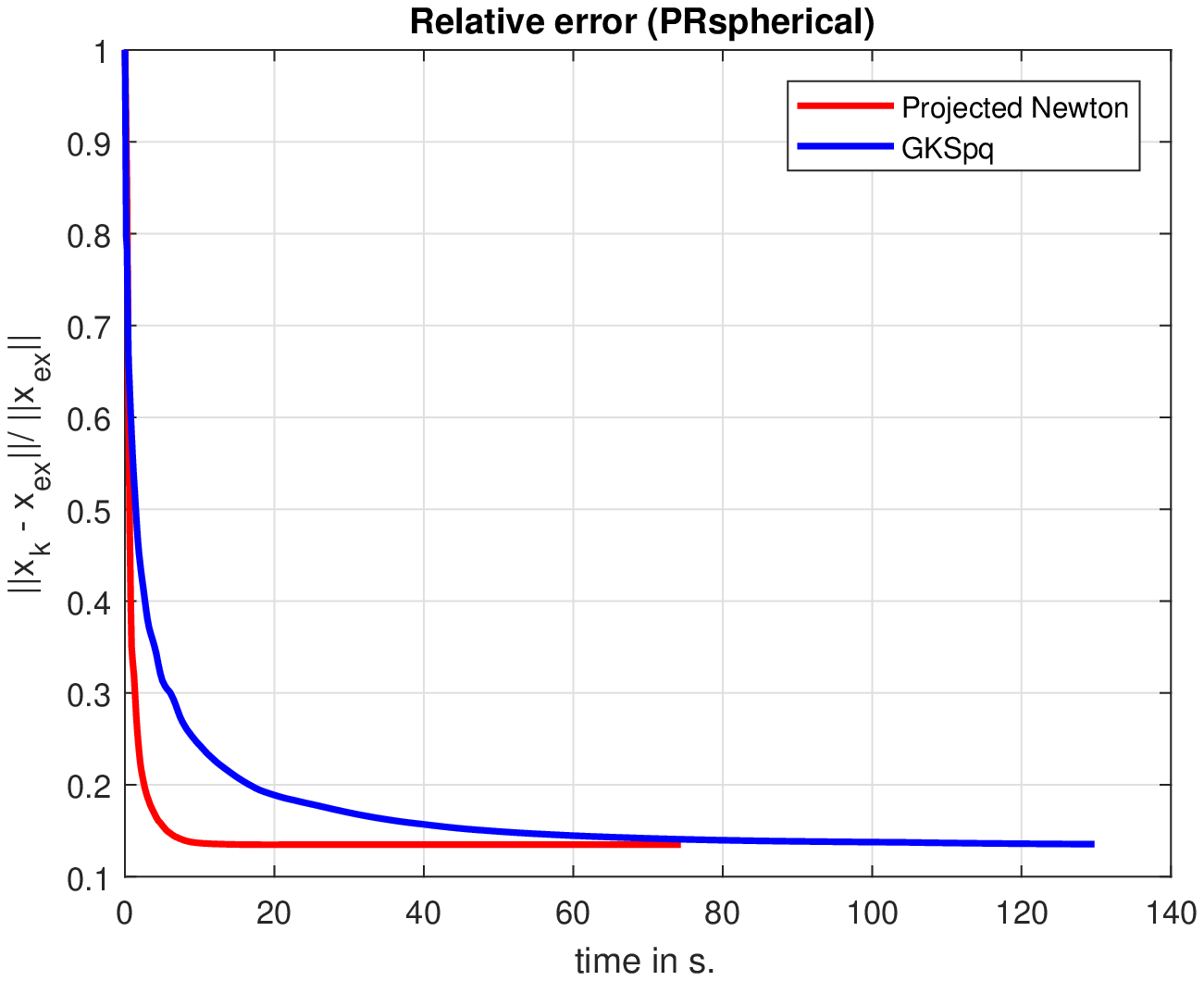} 
\end{tabular}
\caption{\textit{\textbf{Experiment 5.}} Comparison of relative error of the iterates of the Projected Newton method and the GKSpq method \cref{alg:GKSpq} in terms of run-time (in seconds) for $\Psi(x) = \Psi_1(Lx)$ the smooth approximation of anisotropic total-variation function, where $L$ is defined as \cref{eq:ltv}. \label{fig:timingtv}}
\end{figure}

\subsection{Different possible stopping criteria} \label{sec:stopping}

\paragraph{\textbf{Experiment 6.}} For the next numerical experiment we again consider total variation regularization, i.e. we choose $\Psi(x) = \Psi_1(Lx)$ with $L$ given by \cref{eq:ltv}. With this experiment we want to compare different convergence metrics that can possibly be used to formulate a stopping criterion. We already observed in experiment 3 that it is definitely possible that the relative error stabilizes well before $||F(x_k,\lambda_k)||$ is very small.

We again take an example from image deblurring with matrix $A$ generated by the MATLAB function $\texttt{PRblurshake}$ from the IR Tools package and a tomography example with matrix generated by $\texttt{PRtomo}$. For the deblurring test-problem we take the exact image $X\in\mathbb{R}^{256 \times 256}$ as shown in \cref{fig:solutions_tv} and add $20\%$ Gaussian noise to the data. 
%The corresponding deblurring matrix has size $65,536 \times 65,536$ and the regularization matrix has size $130,560 \times 65,536$. 
For the CT test-problem we take the same exact solution and add $1\%$ Gaussian noise to the data. %In this case we get an under-determined matrix A of size $65,160 \times 65,536$.

We stop the algorithm when $||F(x_k,\lambda_k)||<10^{-1}$, set $\beta = 10^{-4}$ and show the reconstructed solution for \texttt{PRblurshake} in \cref{fig:solutions_tv} as well as the solution obtained by standard form Tikhonov regularization as a reference. Next, we show for both test-problems the relative error in \cref{fig:combined}, as well as the following four convergence metrics: 
\begin{enumerate}
\item the norm of the nonlinear equation evaluated in the iterates: $||F(x_k,\lambda_k)||$,
\item the relative difference of the regularization parameter: $|\lambda_{k} - \lambda_{k-1}|/|\lambda_{k-1}|$,
\item the relative difference of the solution:  $||x_{k} - x_{k-1}||/||x_{k-1}||$,
\item the mismatch with the discrepancy principle: $||Ax_{k} - b||-\sigma$. 
\end{enumerate}
All these values can be computed during the iterations without a noticeable computational overhead. Note that it follows from \cref{thm:posdis} that the value used in (iv) is positive. The values $||F(x_k,\lambda_k)||$ form a decreasing sequence since we enforce this by the backtracking line search. Hence this convergence metric behaves quite nicely and smoothly. However, the relative error stagnates well before this value becomes very small. Hence, it seems that it is not necessary to find the actual root of $F(x,\lambda)$ to obtain a good solution to the inverse problem. The relative difference of the regularization parameter decreases much more rapidly, but has also a severely non-monotonic behavior and can exhibit large spikes. The relative difference in the solution $x_k$ seems to be a bit more smooth, although also definitely not monotonic. The mismatch with the discrepancy principle is the value that seems to decrease the most rapidly for these test-problems among all the convergence metrics we have shown. Moreover, it is also smoother than the metrics (ii) and (iii).  What this indicates is that (for these particular test-problems) the Projected Newton method converges quickly to a solution $x_k$ that (approximately) satisfies the discrepancy principle, while it takes much longer to actually satisfy the optimality condition \cref{eq:smooth_cond}. However, it seems that it is not that important for the quality of the reconstruction by observing the relative error. Due to these reasons we suggest to use (iv) as the convergence criteria. Lastly, we mention the very rapid convergence of the Projected Newton method in terms of the relative error. For the deblurring test-problem the relative error stabilizes around iteration $k=150 \ll n = 65,536$, while for the CT test-problem it stabilizes around iteration $k = 50 \ll n = 65,536$. Recall that this is very important for performance of the algorithm because of the term $sk^2$ in \cref{eq:flops}.

\begin{figure}
\begin{center}
\includegraphics[width=1\textwidth]{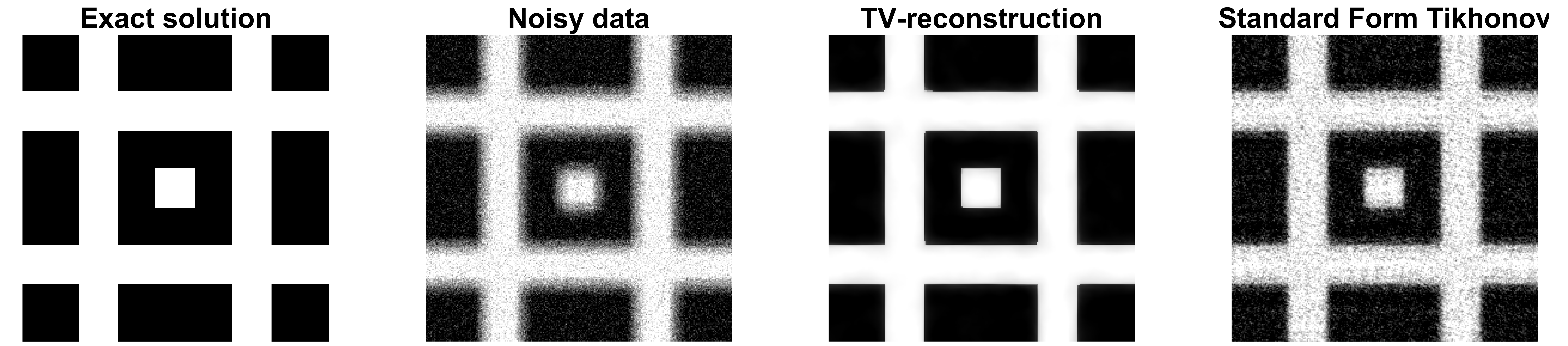} 
\end{center}
\caption{\textit{\textbf{Experiment 6.}} The exact solution, noisy data and reconstructed solution for the image deblurring test-problem $\texttt{PRblurshake}$ as well as the solution obtained by standard form Tikhonov regularization as a reference. \label{fig:solutions_tv}}
\end{figure}

\begin{figure}
\includegraphics[width=1\textwidth]{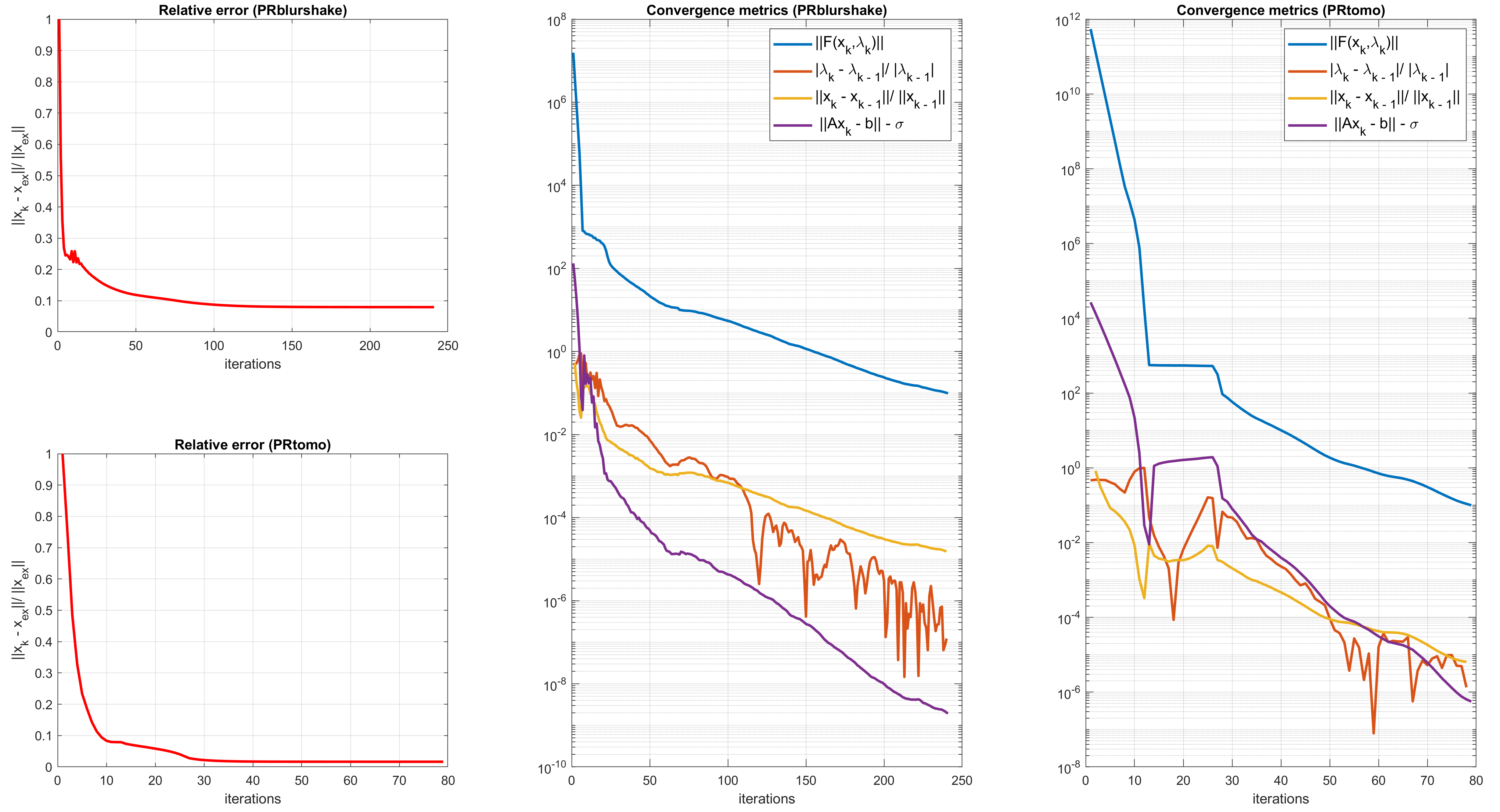} 
\caption{\textit{\textbf{Experiment 6.}} The relative error and other (computable) convergence metrics are shown for the test-problems \texttt{PRblurshake} and \texttt{PRtomo} with $\Psi(x) = \Psi_1(Lx)$ the smooth approximation of anisotropic total-variation function, where $L$ is defined as \cref{eq:ltv}. \label{fig:combined}}
\end{figure}

\paragraph{\textbf{Experiment 7.}} As a final experiment we take a closer look at the different possible convergence criteria and try to answer the question of what threshold can be chosen in practice to obtain an accurate solution, while not performing to many ``unnecessary" iterations. In an ideal world we would terminate the algorithm once the relative error has stagnated. However, the relative error is obviously unobtainable in practice. To study what threshold we can use for the different stopping criteria described above we again consider the four test-problems \texttt{PRblurshake}, \texttt{PRblurrotation}, \texttt{PRtomo} and \texttt{PRspherical} and the exact image of size $256 \times 256$ as illustrated in \cref{fig:solutions_tv}. For all test-problems we consider relative noise-levels $\tilde{\epsilon} = 0.01, 0.03, 0.05$ and $0.1$. For each of these relative noise-levels we run the Projected Newton method 100 times with total variation regularization (using the same parameters as in experiment 6). Note that for each run we generate a new noisy right-hand side $b$. Moreover, for the test-problem \texttt{PRblurshake}, the matrix also changes for each run due to the random nature of the modeled ``shaking'' motion of the camera. The other matrices stay the same over the different runs. We terminate the algorithm once the following condition holds for three consecutive iterations:
\begin{equation}\label{eq:crit}
\frac{|e_k - e_{k-1}|}{e_{k-1}} < 10^{-3}, \hspace{0.5cm}\text{with}\hspace{0.5cm} e_k = \frac{||x_k - x_{ex}||}{||x_{ex}||}.
\end{equation}
This means that the relative error does not change much anymore and that performing additional iterations does not really provide a much better solution. 

In \cref{table} we report the average number of iterations performed until \cref{eq:crit} is satisfied for three consecutive iterations. In addition, we also show the average value obtained by the different stopping criteria in the final iteration $k$ as well as the average relative error. Between parentheses we also report the standard deviation for the different metrics. A few observations can be made. First of all we see that the number of iterations needed to converge is always relatively small compared to size of $A$. Interestingly, we see that the iteration number seems to grow with the relative noise-level. The standard deviation for the number of iterations is also relatively small. As one would expect, we can observe that the relative error in the final iteration also grows with the relative noise-level. All convergence criteria, except $||F(x_k,\lambda_k)||$, are at least smaller than $10^{-3}$. 
Moreover, in most of the cases, we have that the variability over the different runs is modest. 
The mismatch with the discrepancy (i.e. $||Ax_k - b|| - \sigma$) is the value that is the lowest among the different stopping criteria for most cases. However, the experiment does not give a definite answer on what threshold we should use for this convergence criteria. At the moment we argue that a threshold of $10^{-6}$ would nicely balance the accuracy of the solution, while not performing too many ``unnecessary" iterations. The value in the final iteration for the convergence criteria based on the relative change in $\lambda_k$ and $x_k$ seems to be more uniform over the different test-problems. Here, $10^{-4}$ seems to be a good threshold. 

\begin{sidewaystable}
\centering
\begin{tabular}{l | rrrrrrr}
\textbf{Problem} & $\tilde{\epsilon}$ & $\#its$ & $||F(x_k,\lambda_k)||$ & $\frac{|\lambda_k-\lambda_{k-1}|}{\lambda_{k-1}}$ & $\frac{||x_k - x_{k-1}||}{||x_{k-1}||}$ & $||Ax_k-b|| - \eta$  & $e_k$ \\ \hline \hline
\texttt{PRblurshake}    & 0.01  & 57.3  (5.4)    & 2.58 (1.51)  & 2.3e-4 (2.6e-4) & 1.4e-4 (8.0e-5) &  2.7e-6 (3.4e-6) &  1.7e-2 (2.2e-3)  \\
 						            & 0.03  & 69.3  (6.3)    & 2.08 (0.65)  & 1.3e-4 (1.3e-4) & 1.5e-4 (4.0e-5) &  1.1e-6 (1.0e-6) &  3.0e-2 (3.8e-3) \\
						            & 0.05  & 82.8  (7.8)    & 2.02 (0.75)  & 1.0e-4 (2.1e-4) & 1.7e-4 (4.9e-5) &  1.1e-6 (2.0e-6) &  4.0e-2 (4.9e-3) \\
					              & 0.10  & 108.0 (11.3)   & 1.78 (0.87)  & 8.9e-5 (2.5e-4) & 1.9e-4 (5.7e-5) &  9.5e-7 (2.1e-6) &  5.7e-2 (6.7e-3) \\ \hline
\texttt{PRblurrotation} & 0.01  & 131.9 (1.2)    & 2.26 (0.11)  & 3.2e-4 (2.2e-4) & 2.1e-4 (1.0e-5) &  1.0e-6 (1.4e-7) &  4.6e-2 (6.5e-4) \\
 						            & 0.03  & 130.5 (1.7)    & 2.09 (0.17)  & 3.2e-4 (1.1e-4) & 2.3e-4 (1.2e-5) &  4.4e-7 (5.4e-8) &  8.0e-2 (1.4e-3) \\
						            & 0.05  & 140.9 (2.7)    & 2.06 (0.23)  & 3.0e-4 (8.8e-5) & 2.5e-4 (1.4e-5) &  3.8e-7 (5.7e-8) &  1.0e-1 (1.8e-3) \\
					              & 0.10  & 156.6 (4.0)    & 2.10 (0.23)  & 1.2e-4 (7.7e-5) & 3.0e-4 (1.7e-5) &  3.3e-7 (4.6e-8) &  1.4e-1 (3.0e-3) \\ \hline
\texttt{PRtomo}         & 0.01  & 48.5  (0.7)    & 1.95 (0.14)  & 2.5e-4 (7.1e-5) & 1.0e-4 (7.9e-6) &  2.2e-4 (3.5e-5) &  1.7e-2 (2.6e-4) \\
 						            & 0.03  & 59.6  (1.4)    & 2.16 (0.15)  & 1.3e-4 (6.4e-5) & 1.5e-4 (9.9e-6) &  1.5e-4 (2.1e-5) &  3.2e-2 (6.3e-4) \\
						            & 0.05  & 122.0 (3.4)    & 2.18 (0.17)  & 8.0e-5 (5.5e-5) & 1.7e-4 (1.2e-5) &  1.1e-4 (1.4e-5) &  4.3e-2 (1.1e-3) \\
					              & 0.10  & 201.4 (7.5)    & 2.25 (0.30)  & 1.1e-4 (6.4e-5) & 2.2e-4 (2.1e-5) &  9.5e-5 (2.1e-5) &  6.6e-2 (1.9e-3) \\ \hline
\texttt{PRspherical}    & 0.01  & 39.6  (1.0)    & 1.89 (0.22)  & 2.2e-4 (1.3e-4) & 9.9e-5 (1.1e-5) &  1.2e-6 (3.3e-7) &  1.6e-2 (1.8e-4)  \\
 						            & 0.03  & 55.4  (1.4)    & 1.98 (0.14)  & 9.2e-5 (7.0e-5) & 1.3e-4 (8.1e-6) &  7.9e-7 (1.1e-7) &  3.0e-2 (5.1e-4) \\
						            & 0.05  & 69.6  (5.3)    & 1.97 (0.11)  & 5.8e-5 (4.2e-5) & 1.5e-4 (7.5e-6) &  6.5e-7 (7.6e-8) &  4.0e-2 (9.5e-4) \\
					              & 0.10  & 90.5  (2.3)    & 2.07 (0.13)  & 5.6e-5 (3.8e-5) & 1.9e-4 (1.1e-5) &  5.4e-7 (6.2e-8) &  6.1e-2 (1.5e-3) \\ \hline
\end{tabular} 
\\
\caption{\textbf{Experiment 7.} Different convergence metrics of the Projected Newton method applied with total variation regularization to the four test-problems \texttt{PRblurshake}, \texttt{PRblurrotation}, \texttt{PRtomo} and \texttt{PRspherical} with different relative noise-levels $\tilde{\epsilon}$. The Projected Newton method is applied 100 times for each problem and the average value is reported for the number of iterations ($\#its$) needed to satisfy the stopping criteria, for the relative error $e_k$ in the final iteration $k$ and for the different convergence metrics discussed in \cref{sec:stopping}. Between parenthesis we also report the standard deviation over the 100 runs. \label{table}} 
\end{sidewaystable}
\section{Conclusions}\label{sec:concl}
We present a new and efficient Newton-type algorithm to compute (an approximation of) the solution of the $\ell_p$ regularized linear inverse problem with a general, possibly non-square or non-invertible regularization matrix. Simultaneously it determines the corresponding regularization parameter such that the discrepancy principle is satisfied. First, we describe a convex twice continuously differentiable approximation of the $\ell_p$ norm for $p\geq 1$. We then use this approximation to formulate a constrained optimization problem that describes the problem of interest. In iteration $k$ of the algorithm we project this optimization problem on a basis of a Generalized Krylov subspace of dimension $k$. We show that the Newton direction of this projected problem results in a descent direction for the original problem. The step-length is then determined in a backtracking line-search. This approach is efficiently implemented using the reduced QR decomposition of a tall and skinny matrix. Further improvements in the case of general form Tikhonov regularization are presented as well.  

Next, we illustrate the interesting matrix structure of Generalized Krylov subspaces and compare the efficiency of the Projected Newton method with other state-of-the-art approaches in a number of numerical experiments. We show that the Projected Newton method is able to produce a highly accurate solution to an ill-posed linear inverse problem with a modest computational cost. Furthermore, we show that we are successfully able to produce sparse solutions with the approximation of the $\ell_1$ norm and that the approximation of the total variation regularization function also performs well. Lastly, we comment on a number of different convergence criteria that can be used and argue that the mismatch with the discrepancy principle is a suitable candidate. 

\section*{References}
\bibliographystyle{unsrt}
\bibliography{refs_lp}

\end{document}